\documentclass[reqno,12pt]{amsart}
\usepackage{amscd,amsfonts,mathrsfs,amsthm,amssymb,amsmath,mathtools}
\usepackage{upgreek}
\usepackage{csquotes}
\usepackage{enumerate}
\usepackage{bbm}
\usepackage{multicol} 
\usepackage{stmaryrd,color}
\usepackage{array}
\usepackage{ae}
\usepackage{accents}
\usepackage{epsfig}
\usepackage[e]{esvect}
\usepackage[nobysame]{amsrefs}
\usepackage[margin=1in]{geometry}
\usepackage[toc,title,page]{appendix}
\usepackage{hyperref}
\hypersetup{colorlinks=true,linkcolor=blue,filecolor=blue,urlcolor=blue,citecolor=blue}

\let\oldtocsection=\tocsection
\let\oldtocsubsection=\tocsubsection
\let\oldtocsubsubsection=\tocsubsubsection
\renewcommand{\tocsection}[2]{\hspace{0em}\oldtocsection{#1}{#2}}
\renewcommand{\tocsubsection}[2]{\hspace{1em}\oldtocsubsection{#1}{#2}}
\renewcommand{\tocsubsubsection}[2]{\hspace{2em}\oldtocsubsubsection{#1}{#2}}

\def\Xint#1{\mathchoice
 {\XXint\displaystyle\textstyle{#1}}
 {\XXint\textstyle\scriptstyle{#1}}
 {\XXint\scriptstyle\scriptscriptstyle{#1}}
 {\XXint\scriptscriptstyle\scriptscriptstyle{#1}}
 \!\int}
\def\XXint#1#2#3{{\setbox0=\hbox{$#1{#2#3}{\int}$}
 \vcenter{\hbox{$#2#3$}}\kern-.5\wd0}}

\newtheorem{theorem}{Theorem}[section]

\newtheorem{lemma}[theorem]{Lemma}
\newtheorem{corollary}[theorem]{Corollary}
\newtheorem{definition}[theorem]{Definition}
\theoremstyle{definition}
\newtheorem{remark}[theorem]{Remark}

\numberwithin{equation}{section}

\begin{document}

\date{\today}
\thanks{The research project is implemented in the framework of H.F.R.I call ``Basic research Financing (Horizontal support of all Sciences)'' under the National Recovery and Resilience Plan ``Greece 2.0" funded by the European Union - NextGenerationEU (H.F.R.I. Project Number: 14758).}

\title[Strip-type operators and abstract Cauchy problems]{Strip-type operators and abstract Cauchy problems}

\author[N. Roidos]{\textsc{Nikolaos Roidos}}
\address{Nikolaos Roidos\newline
Department of Mathematics\newline
University of Patras\newline
26504 Rio Patras, Greece\newline
{\sl E-mail address:} {\bf roidos@math.upatras.gr}}

\begin{abstract}
We consider the non-homogeneous abstract linear Schr\"odinger and wave equations with zero initial conditions, defined by operators of strip-type and parabola-type in Banach spaces, respectively, and establish the well-posedness of classical solutions in appropriate vector-valued Sobolev-Slobodetskii spaces. We obtain analogous results for two extensions of these equations by replacing the previously mentioned boundedness properties of the associated operators with $R$-boundedness. As an application, we consider an abstract semilinear wave equation and establish the existence and uniqueness of classical solutions to this problem for short times.
\end{abstract}

\maketitle

\section{Introduction}

Let $p\in (1,\infty)$, $T>0$, $X$ be a complex Banach space, $f\in L^{p}(0,T;X)$, and $A:D(A)\rightarrow X$ a closed linear operator. In this article, we consider the problems
\begin{equation}\label{ACLSE}
iu'(t)\mp Au(t)=f(t), \quad t\in (0,T), \quad u(0)=0,
\end{equation}
and
\begin{equation}\label{ACLWE}
u''(t)+A^{2}u(t)=f(t), \quad t\in (0,T), \quad u(0)=u'(0)=0,
\end{equation}
where $(\cdot)'=\partial_{t}$ and the square of $A$ is defined in the usual sense; see Section \ref{Sec2}. In \eqref{ACLSE}, we have extended $A$ to an operator from $L^{p}(0,T;D(A))$ to $L^{p}(0,T;X)$ by $v\mapsto Av(t)$ for almost all $t\in (0,T)$; a similar extension is done for \eqref{ACLWE}. The linear homogeneous analogues of \eqref{ACLSE} and \eqref{ACLWE} are treated using the theory of cosine operator functions and their associated generators. For a comprehensive exposition of this theory, see \cite{ABHN} and the references therein; see also \cite{Fatt} and \cite{VaPi}.

We regard \eqref{ACLSE} and \eqref{ACLWE} as equations involving the sum of two operators, namely 
$$
(\pm iA+B)u=-if \quad \text{and} \quad (A^{2}+B^{2})u=f,
$$
respectively, where $B$ denotes the derivation operator with Dirichlet boundary condition at $t=0$; see Section \ref{Sec3}. In Section \ref{Sec4} we prove well-posedness for \eqref{ACLSE} and \eqref{ACLWE} when $f$ belongs to appropriate vector-valued Sobolev-Slobodetskii spaces; see Theorem \ref{aseth} and Theorem \ref{invA2B2}, respectively. Our proofs are based on functional calculus techniques for sectorial operators and make use of the Da Prato-Grisvard formula for the inverse of the sum of two operators \cite[(3.11)]{DG}. The main assumption for solving \eqref{ACLSE} is that the associated operator $A$ is sectorial, its spectrum lies within a horizontal strip around the real axis, and its resolvent family is either bounded or decays outside this strip. On the other hand, for \eqref{ACLWE}, the associated operator $A^{2}$ is also sectorial, its spectrum lies on the right-hand side of a horizontal parabola opening to the right, and its resolvent family satisfies specific boundedness conditions on the left-hand side of this parabola. 

In Section \ref{Sec5}, we consider two extensions of \eqref{ACLSE} and \eqref{ACLWE}, defined via the closures of the operators $\pm iA+B$, namely \eqref{stabase} and \eqref{stabawe}, respectively. In Theorem \ref{AWEWP} and Theorem \ref{BestScrth}, we restrict to UMD spaces and establish well-posedness of these equations by constructing the inverses of the above closures and of their product in the appropriate spaces. The forcing term $f$ is now allowed to lie in a space that contains the space of solutions. In our proofs, we use the properties of the vector valued Laplace transform and a Mihlin-type Fourier multiplier theorem for operator-valued multiplier functions; see \cite{ABHN} and \cite{Weis}, respectively. The main assumption is now that the associated operator $A$ is sectorial, its spectrum lies within a horizontal strip around the real axis, and its resolvent family is $R$-bounded outside this strip. 

Based on the results of Section \ref{Sec5}, in Section \ref{Sec6} we consider an abstract semilinear wave equation by replacing $f$ in \eqref{ACLWE} with a polynomial in $u$ with vector-valued coefficients; see \eqref{NLW1}-\eqref{NLW2}. Under certain assumptions, using a Banach fixed-point argument, we establish the existence and uniqueness of classical solutions to this problem for short times; see Theorem \ref{ThNLW}. The assumption on the associated operator $A^{2}$ is that it is sectorial, its spectrum lies on the right-hand side of a horizontal parabola opening to the right, and its resolvent family satisfies specific $R$-boundedness conditions on the left-hand side of this parabola. 

As an application, in Section \ref{Sec7} we consider the Klein-Gordon equation on $\mathbb{R}$ and establish existence, uniqueness, and regularity results for its solutions.

 {\em Notation}: Denote $\mathbb{N}_{0}=\mathbb{N}\cup\{0\}$, and let $\mathrm{Re}(\cdot)$, $\mathrm{Im}(\cdot)$, and $\mathrm{arg}(\cdot)$ be the real part, the imaginary part, and the argument, respectively, of a complex number. If $X$ is a Banach space, we denote its norm by $\|\cdot\|_{X}$. Moreover, $\hookrightarrow$ denotes continuously and densely embedding between Banach spaces, and $\mathcal{L}(\cdot,\cdot)$ or $\mathcal{L}(\cdot)$ denotes the space of bounded linear operators acting between Banach spaces. If $\theta\in [0,1]$ and $p\in (1,\infty)$, let $(\cdot,\cdot)_{\theta,p}$ and $[\cdot,\cdot]_{\theta}$ denote the real and the complex interpolation space, respectively, between Banach spaces. Furthermore, $\mathcal{D}(\cdot)$ and $\rho(\cdot)$ denote, respectively, the domain (endowed with the graph norm) and the resolvent set of a linear operator defined in a Banach space. Additionally, if $p\in(1,\infty)\cup\{\infty\}$, $k\in\mathbb{N}_{0}$, and $s\geq0$, let $W^{k,p}$ be the Sobolev space of order $k$, $H_{p}^{s}$ the Bessel potential space of order $s$, and $C^{k}$ the space of $k$-times continuously differentiable functions. Finally, denote by $\Xint-$ the Cauchy principal value. 

\section{sectorial and strip-type operators}\label{Sec2}

Let $X$ be a complex Banach space, and let $A_{1}$, $A_{2}$ be two linear operators in $X$ with domains $\mathcal{D}(A_{1})$ and $\mathcal{D}(A_{2})$, respectively. We define the sum $A_{1}+A_{2}$ and the product $A_{1}A_{2}$ in $X$ as the operators
$$
\mathcal{D}(A_{1})\cap \mathcal{D}(A_{2}) \ni u \mapsto (A_{1}+A_{2})u=A_{1}u+A_{2}u \in X
$$
and
$$
\{v\in \mathcal{D}(A_{2})\, |\, A_{2}v\in \mathcal{D}(A_{1})\}\ni u\mapsto (A_{1}A_{2})u=A_{1}(A_{2}u) \in X.
$$
In particular, for any $k\in \mathbb{N}$ the $(k+1)$-power of $A_{1}$ is defined inductively by 
$$
\mathcal{D}(A_{1}^{k+1})=\{x\in \mathcal{D}(A_{1}^{k}) \, |\, A_{1}^{k}x\in \mathcal{D}(A_{1})\} \quad \text{and} \quad x\mapsto A_{1}^{k+1}x=A_{1}(A_{1}^{k}x).
$$
Recall that, in a Banach space, every linear operator with non empty resolvent set is closed. If $\rho(A_{1}), \rho(A_{2})\neq\emptyset$ and $A_{1}$, $A_{2}$ are resolvent commuting (see, e.g., \cite[(III.4.9.1)]{Amann1} for the notion of two resolvent commuting operators), then $A_{1}+A_{2}$ is closable. Indeed, let $u_{k}\in \mathcal{D}(A_{1})\cap \mathcal{D}(A_{2})$, $k\in\mathbb{N}$, such that $u_{k}\rightarrow 0$ and $(A_{1}+A_{2})u_{k}\rightarrow w$ as $k\rightarrow\infty$. If $\lambda$ in $\rho(A_{1})$, then $(A_{1}+\lambda)^{-1}u_{k}\rightarrow 0$ and $A_{1}(A_{1}+\lambda)^{-1}u_{k}+A_{2}(A_{1}+\lambda)^{-1}u_{k}\rightarrow (A_{1}+\lambda)^{-1}w$ as $k\rightarrow\infty$, i.e. $A_{2}(A_{1}+\lambda)^{-1}u_{k}\rightarrow (A_{1}+\lambda)^{-1}w$ as $k\rightarrow\infty$. Hence, by the closedness of $A_{2}$, we get $(A_{1}+\lambda)^{-1}w=0$, i.e. $w=0$. 

We now continue with the basic notion of a sectorial operator.

\begin{definition}[sectorial operators]
Let $K\geq1$, $\phi\in[0,\pi)$, and let $X$ be a complex Banach space. Let $\mathcal{P}(K,\phi)$ be the class of all closed densely defined linear operators $A$ in $X$ such that 
$$
S_{\phi}=\{\lambda\in\mathbb{C}\,|\, |\arg(\lambda)|\leq\phi\}\cup\{0\}\subset\rho{(-A)} \quad \mbox{and} \quad (1+|\lambda|)\|(A+\lambda)^{-1}\|_{\mathcal{L}(X)}\leq K, \quad \lambda\in S_{\phi}.
$$
The elements in $\mathcal{P}(\phi)=\cup_{K\geq1}\mathcal{P}(K,\phi)$ are called {\em (invertible) sectorial operators of angle $\phi$}. 
\end{definition}

If $A\in\mathcal{P}(\phi)$ for some $\phi\in[0,\pi)$, then there exist $r>0$ and $\varphi\in(\phi,\pi)$ such that
\begin{equation}\label{extsect}
\Omega_{r,\varphi}=\{\lambda\in \mathbb{C}\, |\, |\lambda|\leq r\}\cup S_{\varphi} \subset \rho(-A) \quad \text{and} \quad (1+|\lambda|)\|(A+\lambda)^{-1}\|_{\mathcal{L}(X)}\leq C, \quad \lambda\in \Omega_{r,\varphi},
\end{equation}
for certain $C\geq1$, see, e.g., \cite[(III.4.7.12)-(III.4.7.13)]{Amann1}. Moreover, for any $\rho\geq0$ and $\phi\in(0,\pi)$, consider the upward-oriented path
$$
\Gamma_{\rho,\phi}=\{re^{-i\phi}\in\mathbb{C}\,|\,r\geq\rho\}\cup\{\rho e^{i\varphi}\in\mathbb{C}\,|\,\phi\leq\varphi\leq2\pi-\phi\}\cup\{re^{i\phi}\in\mathbb{C}\,|\,r\geq\rho\}.
$$
Sectorial operators admit a holomorphic functional calculus, which is defined via the Dunford integral formula. A typical example is the complex powers: if $A\in \mathcal{P}(0)$, then for $\mathrm{Re}(z)<0$ they are defined by
\begin{equation}\label{fracpaw}
A^{z}=\frac{1}{2\pi i}\int_{\Gamma_{\rho,\varphi}}(-\lambda)^{z}(A+\lambda)^{-1}d\lambda,
\end{equation}
for certain $\rho>0$ and $\varphi\in (0,\pi)$. If, in particular, $\theta\in (0,1)$, then by Cauchy's theorem we can contract the path $\Gamma_{\rho,\varphi}$ to $[0,\infty)$ in the he above formula to obtain
\begin{equation}\label{fracpow}
A^{-\theta}=\frac{\sin(\pi\theta)}{\pi}\int_{0}^{\infty}s^{-\theta}(A+s)^{-1}ds.
\end{equation}
The family $\{A^{z}\}_{\mathrm{Re}(z)<0}\cup\{A^{0}=I\}$ is a strongly continuous analytic semigroup on $X$; see, e.g., \cite[Theorem III.4.6.2 and Theorem III.4.6.5]{Amann1}. Moreover, each $A^{z}$, with $\mathrm{Re}(z)<0$, is injective, and the complex powers for positive real part, $A^{-z}$, are defined by $A^{-z}=(A^{z})^{-1}$; see, e.g., \cite[(III.4.6.12)]{Amann1}. In addition, after integrating by parts in \eqref{fracpow}, we can define the imaginary powers $A^{it}$, $t\in\mathbb{R}\backslash\{0\}$, to be the closure of the operator
$$
A^{it}=\frac{\sin(i\pi t)}{i\pi t}\int_{0}^{\infty}s^{it}(A+s)^{-2}Ads : \mathcal{D}(A) \rightarrow X,
$$
see, e.g., \cite[(III.4.6.21)]{Amann1}. Assume that that there exist some $\delta,M>0$ such that 
\begin{equation}\label{bipdef1}
A^{it}\in \mathcal{L}(X) \quad \text{and} \quad \|A^{it}\|_{\mathcal{L}(X)}\leq M \quad \text{when} \quad t\in[-\delta,\delta].
\end{equation}
Then, see, e.g., \cite[Theorem III.4.7.1]{Amann1}, we have $A^{it}\in \mathcal{L}(X_{0})$ for all $t\in\mathbb{R}$.

\begin{definition}[bounded imaginary powers] 
Let $X$ be a complex Banach space, and let $A\in\mathcal{P}(0)$ in $X$. If $A$ satisfies \eqref{bipdef1}, then we say that {\em $A$ has bounded imaginary powers} and denote this by $A\in\mathcal{BIP}$.
\end{definition}

For further properties of the complex powers of sectorial operators, we refer to \cite[Section III.4.6 and Section III.4.7]{Amann1} or to \cite[Section 15.2]{HNVW1}. We also recall the following relation between the domain of the fractional powers of a sectorial operator and the corresponding interpolation spaces. If $A\in \mathcal{P}(0)$ in a Banach space $X$, then for any $p,q\in(1,\infty)$ and $0<\xi<\gamma<\nu<\eta<\theta<1$ we have
\begin{equation}\label{intemb}
(X,\mathcal{D}(A))_{p,\theta} \hookrightarrow [X,\mathcal{D}(A)]_{\eta}\hookrightarrow \mathcal{D}(A^{\nu}) \hookrightarrow [X,\mathcal{D}(A)]_{\gamma}\hookrightarrow (X,\mathcal{D}(A))_{q,\xi},
\end{equation}
see, e.g., \cite[(I.2.5.2) and (I.2.9.6)]{Amann1}. 

We recall the following notion of boundedness for a family of operators, which is stronger than the usual notion of uniform boundedness.

\begin{definition}[$R$-boundedness]
In a Banach space $X$, a family $\mathcal{A}\subset \mathcal{L}(X)$ is called {\em $R$-bounded} if there exists a $C>0$ such that for any $n\in\mathbb{N}$, operators $A_{1},\dots,A_{n}\in \mathcal{A}$, and vectors $x_{1},\dots,x_{n}\in X$, we have
$$
\|\sum_{k=1}^{n}\epsilon_{k}A_{k}x_{k}\|_{L^{2}(0,1;X)} \leq C \|\sum_{k=1}^{n}\epsilon_{k}x_{k}\|_{L^{2}(0,1;X)},
$$
where $\{\epsilon_{k}\}_{k\in\mathbb{N}}$ is the sequence of Rademacher functions.
\end{definition} 

Next, we consider certain classes of linear operators in a Banach space whose spectrum is either contained in a strip around the real axis or lies on the right-hand side of a horizontal parabola opening to the right.

\begin{definition}[strip-type operators]
Let $c>0$, and let $X$ be a complex Banach space.\\
{\bf (i)} Let $K>0$, and let $\mathcal{Z}_{c,K}$ be the class of all closed densely defined linear operators in $X$ such that, if $A\in \mathcal{Z}_{c,K}$, then
$$
Z_{c}=\{\lambda\in \mathbb{C}\, | \, |\mathrm{Im}(\lambda)|\geq c\}\subset \rho(-A) \quad \text{and} \quad \|(A+\lambda)^{-1}\|_{\mathcal{L}(X)}\leq K \quad \text{for all} \quad \lambda\in Z_{c}.
$$
The elements in $\mathcal{Z}_{c}=\cup_{K>0}\mathcal{Z}_{c,K}$ are called {\em strip-type operators}.\\
{\bf (ii)} Let $\mathcal{RZ}_{c}$ be the subclass of $\mathcal{Z}_{c}$ such that if $A\in \mathcal{RZ}_{c}$, then the family $\{(A+\lambda)^{-1}\, |\, \lambda\in Z_{c}\}$ is $R$-bounded.\\
{\bf (iii)} Let $\mathcal{S}_{c}$ be the subclass of $\mathcal{Z}_{c}$ such that if $A\in \mathcal{S}_{c}$, then
$$
\|(A+\lambda)^{-1}\|_{\mathcal{L}(X)}\rightarrow 0 \quad \text{as} \quad |\mathrm{Im}(\lambda)|\rightarrow\infty, \quad \lambda\in Z_{c}.
$$
{\bf (iv)} Let $K>0$, and let $\mathcal{V}_{c,K}$ be the subclass of $\mathcal{P}(0)$ such that, if $A\in \mathcal{V}_{c,K}$, then
$$
\Pi_{c}=\Big\{\lambda \in\mathbb{C} \, |\, \mathrm{Re}(\lambda)\geq c^{2}-\frac{(\mathrm{Im}(\lambda))^{2}}{4c^{2}} \Big\}\subset \rho(-A),
$$
$$
 \|(A+\lambda)^{-1}\|_{\mathcal{L}(X)}\leq \frac{K}{\sqrt{|\lambda|}}, \quad \text{and} \quad \|A^{\frac{1}{2}}(A+\lambda)^{-1}\|_{\mathcal{L}(X)}\leq K \quad \text{for all} \quad \lambda\in \Pi_{c}.
$$
Denote $\mathcal{V}_{c}=\cup_{K>0}\mathcal{V}_{c,K}$.\\
{\bf (v)} Let $\mathcal{RV}_{c}$ be the subclass of $\mathcal{V}_{c}$ such that, if $A\in \mathcal{RV}_{c}$, then the families
$$
\{\pm \sqrt{\lambda}(A+\lambda)^{-1} \, |\, \lambda\in \Pi_{c}\}, \quad \{A^{\frac{1}{2}}(A+\lambda)^{-1} \, |\, \lambda\in \Pi_{c}\},
$$
are $R$-bounded. 
\end{definition}

For the theory of strip-type operators, including their functional calculus and some spectral properties, we refer to \cite{BMV}, \cite{Haase2}, \cite{Haase3}, \cite[Chapter 4]{Haase}, and \cite[Section 15.4]{HNVW1}. For instance, examples of strip-type operators include the logarithms of sectorial operators, see, e.g., \cite[Theorem 15.4.3]{HNVW1}. Also, observe that $\mathcal{V}_{c}$ is a subclass of {\em parabola-type operators}; see, e.g., \cite[Definition 3.1]{Haase3} and also compare with \cite[Proposition 3.14.18]{ABHN}.

\begin{remark}\label{StripvsParabola}
For any $c>0$, let the set
\begin{equation}\label{StrVSParareas}
H_{c}=\{\lambda\in \mathbb{C} \, |\, \mathrm{Re}(\lambda)|\geq c\}.
\end{equation}
The map $\lambda \mapsto \lambda^{2}$ maps $H_{c}$ onto $\Pi_{c}$, and the map $z\mapsto \pm\sqrt{z}$ maps $\Pi_{c}$ onto $H_{c}$. Hence, if $X$ is a Banach space and $A$ is a linear operator in $X$, then by the identities 
$$
(iA+\lambda)^{-1}(-iA+\lambda)^{-1}(A^{2}+\lambda^{2})=(iA+\lambda)^{-1}(-iA+\lambda)^{-1}(-iA+\lambda)(iA+\lambda)=I \quad \text{on} \quad \mathcal{D}(A^{2}),
$$
$$
(A^{2}+\lambda^{2})(iA+\lambda)^{-1}(-iA+\lambda)^{-1}=(-iA+\lambda)(iA+\lambda)(iA+\lambda)^{-1}(-iA+\lambda)^{-1}=I \quad \text{on} \quad X,
$$
$$
(iA+\lambda)^{-1}(-iA+\lambda)^{-1}=\frac{1}{2\lambda}((iA+\lambda)^{-1}+(-iA+\lambda)^{-1}) \quad \text{on} \quad X,
$$
$$
A(iA+\lambda)^{-1}(-iA+\lambda)^{-1}=\frac{1}{2i}((-iA+\lambda)^{-1}-(iA+\lambda)^{-1}) \quad \text{on} \quad X,
$$
and
$$
(\mp i A+\lambda)(A^{2}+\lambda^{2})^{-1} (\pm i A+\lambda)=I \quad \text{on} \quad \mathcal{D}(A),
$$
$$
(\pm i A+\lambda)(\mp i A+\lambda)(A^{2}+\lambda^{2})^{-1} =I \quad \text{on} \quad X,
$$
$$
(\mp i A+\lambda)(A^{2}+\lambda^{2})^{-1}=\lambda(A^{2}+\lambda^{2})^{-1}\mp i A(A^{2}+\lambda^{2})^{-1} \quad \text{on} \quad X,
$$
valid for any $\lambda\in H_{c}$, we see that $A\in \mathcal{Z}_{c}$ if and only if
$$
\Pi_{c}\subset \rho(-A^{2}) \quad \text{and} \quad \|(A^{2}+z)^{-1}\|_{\mathcal{L}(X)}\leq \frac{K}{\sqrt{|z|}}, \quad \|A(A^{2}+z)^{-1}\|_{\mathcal{L}(X)}\leq K, \quad z\in \Pi_{c},
$$
for certain $K>0$. Moreover, $A\in \mathcal{RZ}_{c}$ if and only if $\Pi_{c}\subset \rho(-A^{2})$ and the families
$$
\{\pm \sqrt{z}(A^{2}+z)^{-1} \, |\, z\in \Pi_{c}\}, \quad \{A(A^{2}+z)^{-1} \, |\, z\in \Pi_{c}\},
$$
are $R$-bounded. \mbox{\ } \hfill $\Diamond$
\end{remark}

\begin{lemma}\label{fracpowdec}
Let $\theta\in (0,1)$, $X$ be a complex Banach space, and let $A\in \mathcal{P}(0)$ in $X$. For any $\lambda\in \rho(-A)\backslash [0,\infty)$, we have
$$
(A+\lambda)^{-1}A^{-\theta}=(-\lambda)^{-\theta}(A+\lambda)^{-1}+Q_{A}(\lambda),
$$
where 
$$
Q_{A}(z)=\frac{\sin(\pi \theta)}{\pi}\int_{0}^{\infty}\frac{s^{-\theta}}{z-s}(A+s)^{-1}ds \in \mathcal{L}(X), \quad z\in \mathbb{C}\backslash [0,\infty).
$$
Moreover, the map
$$
 \mathbb{C}\backslash [0,\infty)\ni z \mapsto Q_{A}(z)\in \mathcal{L}(X)
$$
is holomorphic, and for any $c>0$ we have
$$
\|Q_{A}(z)\|_{\mathcal{L}(X)}\rightarrow 0 \quad \text{as} \quad |z|\rightarrow \infty, \quad z\in Z_{c}.
$$
Finally, for any $\phi\in (0,\pi]$ and $\rho>0$, there exists a $C>0$ such that 
$$
|z|\|Q_{A}(z)\|_{\mathcal{L}(X)}\leq C\quad \text{for all} \quad z\in \{\lambda\in \mathbb{C}\, |\, |\lambda|\geq\rho \,\, \text{and} \,\, \arg(\lambda)\geq\phi\}.
$$
\end{lemma}
\begin{proof}
By \eqref{fracpow}, we have
\begin{eqnarray*}
(A+\lambda)^{-1}A^{-\theta}&=&\frac{\sin(\pi\theta)}{\pi}\int_{0}^{\infty}\frac{s^{-\theta}}{s-\lambda}((A+\lambda)^{-1}-(A+s)^{-1})ds\\
&=&\Big(\frac{\sin(\pi\theta)}{\pi}\int_{0}^{\infty}\frac{s^{-\theta}}{s-\lambda}ds\Big)(A+\lambda)^{-1}+Q_{A}(\lambda),
\end{eqnarray*}
and the last integral equals $(-\lambda)^{-\theta}$. The rest of the proof follows by the dominated converges theorem (see, e.g., \cite[Theorem 1.1.8]{ABHN}).
\end{proof}

\begin{remark}
Let $c,K>0$, and let $X$ be a complex Banach space.\\
{\bf (i)} By the elementary identity
\begin{equation}\label{resdecDA}
\lambda(A+\lambda)^{-1}=I-A(A+\lambda)^{-1}, \quad \lambda\in \rho(-A),
\end{equation}
if $A\in \mathcal{Z}_{c,K}$ in $X$, then
\begin{equation}\label{decayto0}
|\lambda|\|(A+\lambda)^{-1}\|_{\mathcal{L}(\mathcal{D}(A),X)}\leq K+1 \quad \text{for all} \quad \lambda\in Z_{c}.
\end{equation}
{\bf (ii)} If $A\in \mathcal{P}(0)\cap \mathcal{Z}_{c}$ in $X$, then, by \eqref{extsect}, Lemma \ref{fracpowdec}, and \eqref{decayto0}, for any $\theta\in [0,1]$ there exists a $C>0$ such that 
\begin{equation}\label{BONZ}
|\lambda|^{\theta}\|(A+\lambda)^{-1}A^{-\theta}\|_{\mathcal{L}(X)}\leq C \quad \text{for all} \quad \lambda\in Z_{c}.
\end{equation}
\mbox{\ } \hfill $\Diamond$
\end{remark}

If $a\in\mathbb{R}$, denote by $i\mathbb{R}+a$ the upward-oriented path $\{\lambda\in \mathbb{C}\, | \, \mathrm{Re}(\lambda)=a\}$, where we simply write $i\mathbb{R}$ when $a=0$.

\begin{lemma}\label{sectproj}
Let $a\in\mathbb{R}$, $K>0$, and let $X$ be a complex Banach space. Let $A$ be a linear operator in $X$ such that 
$$
R_{a}=\{\lambda\in \mathbb{C}\, | \, \mathrm{Re}(\lambda) \geq a\}\subset \rho(-A), \quad \text{and} \quad \|(A+\lambda)^{-1}\|_{\mathcal{L}(X)}\leq K \quad \text{for all} \quad \lambda\in R_{a}.
$$
Assume, in addition, that at least one of the following two conditions is satisfied.\\
{\bf (i)} For any $v\in X$, we have
$$
\|(A+\lambda)^{-1}v\|_{X}\rightarrow 0 \quad \text{as} \quad |\lambda| \rightarrow \infty, \quad \lambda\in R_{a}.
$$
{\bf (ii)} We have
$$
\|(A+\lambda)^{-1}\|_{\mathcal{L}(X)}\rightarrow 0 \quad \text{as} \quad \mathrm{Re}(\lambda) \rightarrow \infty, \quad \lambda\in R_{a}.
$$
Then
$$
\Xint-_{i\mathbb{R}+a}(A+\lambda)^{-1}ud\lambda=i\pi u \quad \text{for any} \quad u\in\mathcal{D}(A).
$$
\end{lemma}
\begin{proof}
Starting with
\begin{equation}\label{iprint}
\Xint-_{i\mathbb{R}+a}(1-a+\lambda)^{-1}d\lambda=i\pi,
\end{equation}
we formally have 
$$
\Xint-_{i\mathbb{R}+a}(A+\lambda)^{-1}ud\lambda-i\pi u=\Xint-_{i\mathbb{R}+a}(1-a+\lambda)^{-1}(A+\lambda)^{-1}(1-a-A)ud\lambda.
$$
If (i) holds, then the integral on the right-hand side of the above equation is zero by Cauchy's theorem. If (ii) holds, then, by Cauchy's theorem, the norm of the right-hand side of the above equation becomes
$$
\leq \lim_{r\rightarrow\infty}\int_{-\frac{\pi}{2}}^{\frac{\pi}{2}} \frac{r}{|1+re^{i\phi}|}\|(A+a+re^{i\phi})^{-1}\|_{\mathcal{L}(X)}\|(1-a-A)u\|_{X}d\phi=0,
$$
where, at the last step, we have used the dominated converges theorem. 
\end{proof}

We also obtain the following analogous result.

\begin{lemma}\label{specjct2j}
Let $c>0$, $X$ be a complex Banach space, and let $A$ be a linear operator in $X$. Assume that at least one of the following two conditions is satisfied:\\
{\bf (i)} $A\in \mathcal{P}(0)\cap\mathcal{Z}_{c}$ in $X$ and $u\in \mathcal{D}(A^{1+\theta})$, for some $\theta\in (0,1)$.\\
{\bf (ii)} $A\in \mathcal{S}_{c}$ in $X$ and $u\in \mathcal{D}(A)$.\\
Then
$$
\Xint-_{i\mathbb{R}+c}(\pm iA+\lambda)^{-1}ud\lambda =i\pi u.
$$
\end{lemma}
\begin{proof}
{\bf (i)} 
By \eqref{iprint}, we formally have
$$
\Xint-_{i\mathbb{R}+c}(\pm iA+\lambda)^{-1}ud\lambda -i\pi u=\Xint-_{i\mathbb{R}+c}(1-c+\lambda)^{-1}(\pm iA+\lambda)^{-1}A^{-\theta}(1-c\mp iA)A^{\theta}ud\lambda,
$$
and the last integral is zero by \eqref{BONZ} and Cauchy's theorem.\\
{\bf (ii)} The result follows by Lemma \ref{sectproj}.
\end{proof}

\section{The derivation operator in $L^{p}$}\label{Sec3}

In this section, we study the derivation operator on vector-valued functions defined over a finite interval. We begin with the notion of Sobolev spaces of fractional order, see, e.g., \cite[Remark VII.3.6.4]{Amann2} or \cite[Section 4.1]{DG} or \cite[Definition 2.5.16]{HNVW}.

\begin{definition}[Sobolev-Slobodetskii spaces]
Let $p\in (1,\infty)$, $k\in \mathbb{N}_{0}$, $s\in (0,1)$, $T>0$, and let $X$ be a Banach space. \\
{\bf (i)} Let $W^{k+s,p}(0,T;X)$ be the space of all functions $u\in W^{k,p}(0,T;X)$ such that 
$$
[\partial^{k}u]_{W^{s,p}(0,T;X)}=\Big(\int_{0}^{T}\int_{0}^{T}\frac{\|(\partial^{k}u)(x)-(\partial^{k}u)(y)\|_{X}^{p}}{|x-y|^{1+sp}}dxdy\Big)^{\frac{1}{p}}<\infty,
$$
endowed with the norm
$$
\|u\|_{W^{s,p}(0,T;X)}=\|u\|_{W^{k,p}(0,T;X)}+[\partial^{k}u]_{W^{s,p}(0,T;X)}.
$$
{\bf (ii)} If $q\in(1,\infty)\cup\{\infty\}$ and $m\in\mathbb{N}$, we denote
$$
W_{0}^{m,q}(0,T;X)=\{u\in W^{m,q}(0,T;X) \, | \, u(0)=(\partial u)(0)=\cdots=(\partial^{m-1}u)(0)=0\}
$$
and
$$
W_{0}^{s,p}(0,T;X)=\Bigg\{\begin{array}{lll} W^{s,p}(0,T;X) & \text{if} & s<1/p, \\ \{u\in W^{s,p}(0,T;X) \, |\, [u]_{p}<\infty\} & \text{if} & s=1/p, \\ \{u\in W^{s,p}(0,T;X) \, |\, u(0)=0\} & \text{if} & s>1/p, \end{array}
$$
where 
$$
[u]_{p}=\Big(\int_{0}^{T}\|u(x)\|_{X}^{p}\frac{dx}{x}\Big)^{\frac{1}{p}}.
$$
Moreover, let
$$
W_{0}^{m+s,p}(0,T;X)=\{u\in W_{0}^{m,p}(0,T;X) \, | \, \partial^{m}u\in W_{0}^{s,p}(0,T;X)\}.
$$
\end{definition}

Let $p\in (1,\infty)$, $T\in (0,\infty)$, and let $X$ be a complex Banach space. Consider the operator
\begin{equation}\label{Bdef}
W_{0}^{1,p}(0,T;X)=D(B)\ni u\mapsto Bu=\partial_{t}u\in L^{p}(0,T;X),
\end{equation}
where the differentiation is understood to be in the sense of almost every $t\in(0,T)$. The spectrum of this operator is empty, and for any $u\in L^{p}(0,T;X)$, the resolvent is given by 
\begin{equation}\label{expresdt}
\big((B+\lambda)^{-1}u\big)(t)=\int_{0}^{t}e^{\lambda(x-t)}u(x)dx \quad t\in[0,T), \quad \lambda\in\mathbb{C}.
\end{equation}
By Young's inequality for convolution, see, e.g., \cite[Proposition 1.3.2(a)]{ABHN} or \cite[Lemma 14.2.3]{HNVW1}, we find that
\begin{equation}\label{BRND}
\|(B+\lambda)^{-1}\|_{\mathcal{L}(L^{p}(0,T;X))}\leq \bigg\{\begin{array}{cll} (1-e^{-\mathrm{Re}(\lambda)T})/(\mathrm{Re}(\lambda)) &, & \lambda\in \mathbb{C}\backslash i\mathbb{R}, \\ T &, & \lambda\in i\mathbb{R}. \end{array}
\end{equation}
Furthermore, by \eqref{resdecDA}, we infer that 
\begin{equation}\label{BDomdecay}
|\lambda|\|(B+\lambda)^{-1}\|_{\mathcal{L}(\mathcal{D}(B),L^{p}(0,T;X))}\leq 1+\bigg\{\begin{array}{cll} (1-e^{-\mathrm{Re}(\lambda)T})/(\mathrm{Re}(\lambda)) &, & \lambda\in \mathbb{C}\backslash i\mathbb{R}, \\ T &, & \lambda\in i\mathbb{R}. \end{array}
\end{equation}

\begin{lemma}\label{decayBres}
Let $p\in (1,\infty)$, $T>0$, and let $X$ be a complex Banach space. If $B$ is the derivation operator defined in \eqref{Bdef}, then, for any $a\in\mathbb{R}$ and $u\in L^{p}(0,T;X)$, we have
$$
\lim_{|\lambda|\rightarrow \infty, \, \mathrm{Re}(\lambda)\geq a}\|(B+\lambda)^{-1}u\|_{L^{p}(0,T;X)}=0.
$$
\end{lemma}
\begin{proof}
Assuming the contrary, there exist $c>0$ and $\lambda_{k}\in \{\lambda\in \mathbb{C} \, | \, \mathrm{Re}(\lambda)\geq a\}$, $k\in \mathbb{N}$, such that $|\lambda_{k}|\rightarrow \infty$ as $k\rightarrow\infty$ and $\|(B+\lambda_{k})^{-1}u\|_{L^{p}(0,T;X)}\geq c$ for all $k\in\mathbb{N}$. If there exists a subsequence $\{\mu_{k}\}_{k\in\mathbb{N}}$ of $\{\lambda_{k}\}_{k\in\mathbb{N}}$ such that $\mathrm{Re}(\mu_{k})\rightarrow \infty$ as $k\rightarrow \infty$, then we obtain a contradiction by \eqref{BRND}. Hence, there exists a $b\geq a$ such that $\mathrm{Re}(\lambda_{k})\leq b$ for each $k\in\mathbb{N}$. Choose a subsequence $\{z_{k}\}_{k\in\mathbb{N}}$ of $\{\lambda_{k}\}_{k\in\mathbb{N}}$ such that $\mathrm{Re}(z_{k})\rightarrow x\in[a,b]$ as $k\rightarrow \infty$. Let $z_{k}=x_{k}+iy_{k}$ with $x_{k},y_{k}\in \mathbb{R}$, $k\in\mathbb{N}$. For any fixed $t\in(0,T]$, we have
$$
\|((B+z_{k})^{-1}u)(t)-((B+x+iy_{k})^{-1}u)(t)\|_{X}\leq \int_{0}^{t}|e^{x_{k}(s-t)}-e^{x(s-t)}|\|u(s)\|_{X}ds
$$
and, by the dominated convergence theorem, the last term tends to zero as $k\rightarrow \infty$. Hence,
\begin{eqnarray*}
\lim_{k\rightarrow\infty}\|((B+z_{k})^{-1}u)(t)-((B+x+iy_{k})^{-1}u)(t)\|_{X}=0.
\end{eqnarray*}
On the other hand, if $\chi_{t}$ is the characteristic function on $[0,t]$, we have
$$
((B+x+iy_{k})^{-1}u)(t)=e^{-(x+iy_{k})t}\int_{\mathbb{R}}e^{iy_{k}s}\chi_{t}(s)e^{xs}u(s)ds,
$$
so that, by the Riemann-Lebesgue lemma (see, e.g., \cite[Section III.4.2, p. 130]{Amann1} or \cite[Theorem 1.8.1(c)]{ABHN}), we find
$$
\lim_{k\rightarrow\infty}\|((B+x+iy_{k})^{-1}u)(t)\|_{X}=0.
$$
We conclude that
$$
\lim_{k\rightarrow \infty}\|((B+z_{k})^{-1}u)(t)\|_{X}=0
$$
for each $t\in[0,T]$. Hence, by the dominated convergence theorem, we obtain
$$
\lim_{k\rightarrow \infty}\|(B+z_{k})^{-1}u\|_{L^{p}(0,T;X)}=0
$$
which yields a contradiction.
\end{proof}

If $s\in (0,1)$, then
\begin{equation}\label{inrtSob0}
(L^{p}(0,T;X), W_{0}^{1,p}(0,T;X))_{s,p}=W_{0}^{s,p}(0,T;X),
\end{equation}
see, e.g., \cite[Appendix A.2]{DG} or \cite[Section 2]{MeSc} (see also \cite[Theorem 2.5.17]{HNVW}). In particular, for any $k\in\mathbb{N}_{0}$, both $W^{k+s,p}(0,T;X)$ and $W_{0}^{k+s,p}(0,T;X)$ are Banach spaces. Note that the definition of $W_{0}^{s,p}(0,T;X)$ when $s>1/p$ makes sense due to the Sobolev embedding
$$
W^{s,p}(0,T;X)\hookrightarrow C([0,T];X),
$$
see, e.g., \cite[Proposition 2.10]{MeSc} or \cite[Section 3.4.5(v)]{PS1} combined with \eqref{intemb} and \eqref{BRND}. Moreover, for any $m \in\mathbb{N}$ we have 
$$
W_{0}^{m,p}(0,T;X)=\mathcal{D}(B^{m})=\{B^{-m}v\, |\, v\in L^{p}(0,T;X)\}.
$$
Finally, by \eqref{intemb}, \eqref{BRND}, and \eqref{inrtSob0}, we have
\begin{eqnarray}\label{BYfracemb}
\lefteqn{W_{0}^{k+s+\varepsilon,p}(0,T;X)=\{B^{-k} v \, |\, v\in W_{0}^{s+\varepsilon,p}(0,T;X)\} }\\\nonumber
&&\hookrightarrow \mathcal{D}(B^{k+s})=\{B^{-k} v \, |\, v\in \mathcal{D}(B^{s})\}\hookrightarrow \{B^{-k} v \, |\, v\in W_{0}^{s-\varepsilon,p}(0,T;X)\}=W_{0}^{k+s-\varepsilon,p}(0,T;X)
\end{eqnarray}
for all $\varepsilon>0$ sufficiently small. 

\section{the abstract schr\"odinger and wave equations}\label{Sec4}

In this section, we present two results on the existence and uniqueness of classical solutions of \eqref{ACLSE} and \eqref{ACLWE} when $f$ lies in appropriate mixed-regularity spaces. 

\begin{theorem}[abstract Schr\"odinger equation]\label{aseth}
Let $p\in (1,\infty)$, $c,T>0$, $s,\ell\in(1,2]$, $X$ be a complex Banach space, $A$ a linear operator in $X$, and let $B$ be the derivation operator defined in \eqref{Bdef}.\\
{\bf (i)} Assume that at least one of the following two conditions is satisfied:\\
{\bf (a)} $A\in \mathcal{P}(\frac{\pi}{2})\cap \mathcal{Z}_{c}$ and $v\in W_{0}^{1,p}(0,T;X)\cap L^{p}(0,T;[X,\mathcal{D}(A^{2})]_{\ell/2})$. \\
{\bf (b)} $A\in \mathcal{S}_{c}$ and $v\in W_{0}^{1,p}(0,T;X)\cap L^{p}(0,T;\mathcal{D}(A))$.\\
Then
$$
J_{\pm}(\pm iA+B)v=v
$$
where $J_{\pm}$ is defined by
\begin{equation}\label{Jdef}
g\mapsto J_{\pm}g=\frac{1}{2\pi i}\Xint-_{i\mathbb{R}-c}(\pm iA-\lambda)^{-1}(B+\lambda)^{-1}g d\lambda.
\end{equation}
{\bf (ii)} Let $A\in \mathcal{P}(0)\cap \mathcal{Z}_{c}$, and choose $r,\gamma\in(0,1]$ such that $r+\gamma>1$. Then, for any $w\in W_{0}^{r,p}(0,T;[X,\mathcal{D}(A)]_{\gamma})$, the integral 
$$
\int_{i\mathbb{R}-c}(\pm iA-\lambda)^{-1}(B+\lambda)^{-1}wd\lambda
$$
converges absolutely with respect to the $L^{p}(0,T;X)$-norm; in particular,
\begin{equation}\label{bop1}
J_{\pm}\in \mathcal{L}( W_{0}^{r,p}(0,T;[X,\mathcal{D}(A)]_{\gamma}),L^{p}(0,T;X)).
\end{equation}
{\bf (iii)} Let $A\in \mathcal{P}(\frac{\pi}{2})\cap \mathcal{Z}_{c}$, and choose $\eta\in(0,s)$, $\nu\in(0,\ell)$ such that $\eta+\nu<s+\ell-1$. Then
\begin{eqnarray}\label{bregop2}
J_{\pm}\in \mathcal{L}(W_{0}^{s,p}(0,T;[X,\mathcal{D}(A^{2})]_{\ell/2}),W_{0}^{\eta,p}(0,T;[X,\mathcal{D}(A^{2})]_{\nu/2}));
\end{eqnarray}
in particular,
\begin{eqnarray}\label{lagoon2}
J_{\pm}\in \mathcal{L}(W_{0}^{s,p}(0,T;[X,\mathcal{D}(A^{2})]_{\ell/2}), W_{0}^{1,p}(0,T;X)\cap L^{p}(0,T;\mathcal{D}(A))).
\end{eqnarray}
Moreover, for any $f\in W_{0}^{s,p}(0,T;[X,\mathcal{D}(A^{2})]_{\ell/2})$, we have $(\pm iA+B)J_{\pm}f=f$.\\
{\bf (iv)} If $A\in \mathcal{P}(\frac{\pi}{2})\cap \mathcal{Z}_{c}$, then for any $f\in W_{0}^{s,p}(0,T;[X,\mathcal{D}(A^{2})]_{\ell/2})$, there exists a $u\in W_{0}^{1,p}(0,T;X) \cap L^{p}(0,T;\mathcal{D}(A))$ satisfying \eqref{ACLSE}; the solution $u$ depends continuously on $f$ and is given by the formula
\begin{equation}\label{expsolASE}
u=\frac{1}{2\pi}\int_{i\mathbb{R}-c}\int_{0}^{t}e^{\lambda(x-t)}((\mp iA+\lambda)^{-1}f)(x)dx d\lambda, \quad t\in[0,T).
\end{equation}
If, in addition, $A\in\mathcal{S}_{c}$, then $u$ is unique.
\end{theorem}
\begin{proof}
If $A\in \mathcal{P}(0)$, then for any $\xi,\rho>0$ and $\varepsilon\in (0,\max\{\xi,\rho\})$, let the Banach space
$$
Z_{\xi,\rho}^{\varepsilon}=\{A^{\varepsilon-\rho}B^{\varepsilon-\xi}x \, | \, x\in L^{p}(0,T;X)\}
$$
be endowed with the norm
$$
y\mapsto \|y\|_{Z_{\xi,\rho}^{\varepsilon}}=\|A^{\rho-\varepsilon}B^{\xi-\varepsilon}y\|_{L^{p}(0,T;X)}.
$$
If $A\in \mathcal{P}(\frac{\pi}{2})$, then, by \eqref{extsect}, there exists a $\varphi\in (\pi/2,\pi)$ such that $A\in \mathcal{P}(\varphi)$. Hence, by \cite[Theorem 15.2.7]{HNVW1} or \cite[Lemma 3.6]{RS3} we obtain that $A^{2}\in \mathcal{P}(0)$, which, combined with \eqref{intemb}, implies 
\begin{equation}\label{A2domemb}
\mathcal{D}(A^{\ell+2\delta})=\mathcal{D}((A^{2})^{\frac{\ell}{2}+\delta}) \hookrightarrow [X,\mathcal{D}(A^{2})]_{\ell/2}\hookrightarrow \mathcal{D}((A^{2})^{\frac{\ell}{2}-\delta})=\mathcal{D}(A^{\ell-2\delta})
\end{equation}
for all $\delta>0$, where in the above equalities we used \cite[Theorem 15.2.7]{HNVW1} once more.

{\bf (i)} By Lemma \ref{specjct2j}(i) in case of (a) and Lemma \ref{specjct2j}(ii) in case of (b), combined with \eqref{A2domemb}, we have
$$
\Xint-_{i\mathbb{R}+c}(\pm iA+\lambda)^{-1}vd\lambda=i\pi v.
$$
Also, Lemma \ref{sectproj}, together with \eqref{BRND} or Lemma \ref{decayBres}, imply
$$
\Xint-_{i\mathbb{R}-c}(B+\lambda)^{-1}vd\lambda=i\pi v.
$$
Therefore,
\begin{eqnarray*}
u&=&\frac{1}{2\pi i}\Xint-_{i\mathbb{R}+c}(\pm iA+\lambda)^{-1}vd\lambda+\frac{1}{2\pi i}\Xint-_{i\mathbb{R}-c}(B+\lambda)^{-1}vd\lambda\\
&=&\frac{1}{2\pi i}\Xint-_{i\mathbb{R}-c}(\pm iA-\lambda)^{-1}(B+\lambda)^{-1}(\pm iA-\lambda+B+\lambda)vd\lambda.
\end{eqnarray*}

{\bf (ii)} For $\varepsilon>0$ sufficiently small and any $h\in Z_{r,\gamma}^{\varepsilon}$, we have
\begin{equation}\label{intJconvabs}
J_{\pm}h=\mp\frac{1}{2\pi}\int_{i\mathbb{R}-c}(A\pm i\lambda)^{-1}A^{-\gamma+\varepsilon}(B+\lambda)^{-1}B^{-r+\varepsilon}B^{r-\varepsilon}A^{\gamma-\varepsilon}hd\lambda,
\end{equation}
and the integral converges absolutely by Lemma \ref{fracpowdec}, \eqref{BONZ}, and \eqref{BRND}; in particular, 
\begin{equation}\label{BoundMaptoZ}
J_{\pm}\in \mathcal{L}(Z_{r,\gamma}^{\varepsilon}, L^{p}(0,T;X)).
\end{equation}
Moreover, by \eqref{intemb} and \eqref{BYfracemb}, we have 
\begin{equation}\label{emb234}
W_{0}^{r,p}(0,T;[X,\mathcal{D}(A)]_{\gamma}) \hookrightarrow W_{0}^{r,p}(0,T;\mathcal{D}(A^{\gamma-\varepsilon})) \hookrightarrow Z_{r,\gamma}^{\varepsilon}.
\end{equation}
Hence, the result follows by \eqref{intJconvabs}, \eqref{BoundMaptoZ}, and \eqref{emb234}.

{\bf (iii)} By \eqref{BYfracemb} and \eqref{A2domemb}, we obtain 
\begin{equation}\label{powg1}
W_{0}^{s,p}(0,T;[X,\mathcal{D}(A^{2})]_{\ell/2})\hookrightarrow W_{0}^{s,p}(0,T;\mathcal{D}(A^{\ell-\varepsilon})) \hookrightarrow Z_{s,\ell}^{\varepsilon}
\end{equation}
for all $\varepsilon>0$ sufficiently small. For any $\eta<\nu_{0}<s$ and $\nu<\eta_{0}<\ell$ such that $\eta_{0}+\nu_{0}<s+\ell-1$, the map 
$$
A^{\eta_{0}}B^{\nu_{0}} : Z_{s,\ell}^{\varepsilon} \rightarrow Z_{s-\nu_{0},\ell-\eta_{0}}^{\varepsilon}
$$
is bounded, and, by \eqref{intJconvabs}, for any $\psi\in Z_{s,\ell}^{\varepsilon}$ the integrals $J_{\pm}\psi$ and $J_{\pm}A^{\eta_{0}}B^{\nu_{0}}\psi$ converges absolutely. Thus, by \eqref{powg1}, we obtain
\begin{equation}\label{fstpbopJ}
J_{\pm}\in \mathcal{L}(W_{0}^{s,p}(0,T;[X,\mathcal{D}(A^{2})]_{\ell/2}), Z_{\eta_{0},\nu_{0}}^{\varepsilon}).
\end{equation}
Furthermore, by \eqref{BYfracemb} and \eqref{A2domemb}, we find
$$
Z_{\eta_{0},\nu_{0}}^{\varepsilon} \hookrightarrow W_{0}^{\eta,p}(0,T;[X,\mathcal{D}(A^{2})]_{\nu/2}),
$$
which, together with \eqref{fstpbopJ}, implies \eqref{bregop2}. In particular, if we choose $(\eta,\nu)=(s-\varepsilon,\varepsilon)$ and $(\eta,\nu)=(\varepsilon,\ell-\varepsilon)$, we deduce \eqref{lagoon2} from \eqref{BYfracemb} and \eqref{A2domemb}. Moreover, by \eqref{powg1}, $A$ and $B$ induce bounded maps from $W_{0}^{s,p}(0,T;[X,\mathcal{D}(A^{2})]_{\ell/2})$ to $Z_{s,\ell-1}^{\varepsilon} $ and $Z_{s-1,\ell}^{\varepsilon}$, respectively, so that, in view of \eqref{intJconvabs}, the integrals $J_{\pm}Af$ and $J_{\pm}Bf$ converge absolutely. Therefore, $AJ_{\pm}f=J_{\pm}Af$ and $BJ_{\pm}f=J_{\pm}Bf$, which, by {\bf (i)}, implies 
$$
(\pm iA+B)J_{\pm}f=J_{\pm}(\pm iA+B)f=f.
$$

{\bf (iv)} The proof follows by (i) and (iii); the expression \eqref{expsolASE} follows from \eqref{expresdt}.
\end{proof}

We continue with a mixed derivative regularity result that will be needed in the sequel. We recall the class of UMD Banach spaces, i.e. Banach spaces that possess the {\em unconditionality of martingale differences property}; see \cite[Section III.4.4 and Section III.4.5]{Amann1} and \cite[Section 4]{HNVW} for the definition and basic properties of these spaces.

\begin{lemma}\label{fracderiv}
Let $p\in (1,\infty)$, $T>0$, and let $X$ be a UMD complex Banach space. If $A\in \mathcal{BIP}$ in $X$, then
$$
W_{0}^{2,p}(0,T;X)\cap L^{p}(0,T;\mathcal{D}(A^{2}))\hookrightarrow W_{0}^{1,p}(0,T;\mathcal{D}(A)).
$$
\end{lemma}
\begin{proof}
Since $A$ has a bounded inverse, $\mathcal{D}(A)$ is isomorphic to $X$, so that, by \cite[Theorem III.4.5.2]{Amann1} it is also a UMD Banach space. By \cite[Theorem VII.7.4.2]{Amann2} and \cite[Theorem 15.3.9]{HNVW1}, we have
\begin{eqnarray*}
W^{2,p}(0,T;X)\cap L^{p}(0,T;\mathcal{D}(A^{2})) & \hookrightarrow & W^{1,p}(0,T;[X,\mathcal{D}(A^{2})]_{\frac{1}{2}}) \\
&=& W^{1,p}(0,T;\mathcal{D}(A)) \quad \text{(up to norm equivalence)},
\end{eqnarray*}
which shows the required embedding. 
\end{proof}

If we consider the composition $J_{+}J_{-}$ with the improper integrals defined on different vertical paths, then we can obtain for the problem \eqref{ACLWE} a result similar to Theorem \ref{aseth}, as follows.

\begin{theorem}[abstract wave equation]\label{invA2B2}
Let $p\in (1,\infty)$, $c,T>0$, $X$ be a complex Banach space, $A$ a linear operator in $X$, and let $B$ be the derivation operator defined in \eqref{Bdef}.\\
{\bf (i)} Let $A\in \mathcal{Z}_{c}$, and for any $r>c$, denote 
$$
G_{\lambda,z}=(A-i\lambda)^{-1}(A+iz)^{-1}(B+\lambda)^{-1}(B+z)^{-1}, \quad \lambda \in i\mathbb{R}-c, \, z\in i\mathbb{R}-r.
$$
Then, for any $v\in W_{0}^{2,p}(0,T;X) \cap L^{p}(0,T;\mathcal{D}(A^{2}))\cap W_{0}^{1,p}(0,T;\mathcal{D}(A))$, we have
\begin{equation}\label{invAB}
v=\frac{1}{(2\pi i)^{2}}\Xint-_{i\mathbb{R}-r}\Xint-_{i\mathbb{R}-c}G_{\lambda,z}(B^{2}+A^{2})vd\lambda dz.
\end{equation}
{\bf (ii)} Let $A\in \mathcal{P}(0)\cap\mathcal{Z}_{c}$ and $\nu\in (0,1)$. For any $w\in W_{0}^{\nu,p}(0,T;X) \cup L^{p}(0,T;[X,\mathcal{D}(A)]_{\nu})$, the integral
\begin{equation}\label{intprop}
\int_{i\mathbb{R}-c}(A^{2}+\lambda^{2})^{-1}(B+\lambda)^{-1}wd\lambda
\end{equation}
converges absolutely with respect to the $L^{p}(0,T;X)$-norm; in particular, the map 
$$
\eta \mapsto L\eta=\int_{i\mathbb{R}-c}(A^{2}+\lambda^{2})^{-1}(B+\lambda)^{-1}\eta d\lambda
$$ 
satisfies 
\begin{equation}\label{Lmapbounded}
L\in \mathcal{L}(W_{0}^{\nu,p}(0,T;X) , L^{p}(0,T;X)) \cap \mathcal{L}(L^{p}(0,T;[X,\mathcal{D}(A)]_{\nu}), L^{p}(0,T;X)).
\end{equation}
Furthermore, if $s,\ell\in(0,1]$ are such that $s+\ell>1$, then, for any $g\in W_{0}^{s,p}(0,T;[X,\mathcal{D}(A)]_{\ell})$, we have
\begin{equation}\label{inteq}
Lg=\frac{1}{2\pi i}\Xint-_{i\mathbb{R}-r}\int_{i\mathbb{R}-c}G_{\lambda,z}gd\lambda dz,
\end{equation}
where the last integral converges absolutely with respect to the $L^{p}(0,T;X)$-norm.\\
{\bf (iii)} Let $A\in \mathcal{P}(0)\cap \mathcal{Z}_{c}$, and assume that at least one of the following two conditions is satisfied:\\
{\bf (a)} $f\in W_{0}^{3,p}(0,T;\mathcal{D}(A))) \cap W_{0}^{1,p}(0,T;\mathcal{D}(A^{3}))$.\\
{\bf (b)} $X$ is UMD, $A\in \mathcal{BIP}$, and $f\in W_{0}^{2+s,p}(0,T;[X,\mathcal{D}(A)]_{\ell}) \cap W_{0}^{s,p}(0,T;[\mathcal{D}(A^{2}),\mathcal{D}(A^{3})]_{\ell})$ for some $s,\ell\in(0,1]$ such that $s+\ell>1$.\\
Then, there exists $u\in W_{0}^{2,p}(0,T;X) \cap L^{p}(0,T;\mathcal{D}(A^{2}))$ satisfying \eqref{ACLWE}; the solution $u$ depends continuously on $f$ and is given by the formula
\begin{equation}\label{expsolAWE}
u(t)=\frac{1}{2\pi i}\int_{i\mathbb{R}-c}\int_{0}^{t}e^{\lambda(x-t)}((A^{2}+\lambda^{2})^{-1}f)(x)dxd\lambda, \quad t\in [0,T).
\end{equation}
Moreover, $u$ is unique if {\em (b)} holds.
\end{theorem}
\begin{proof}
Note that the integration in \eqref{intprop} and \eqref{expsolAWE} is well-defined due to Remark \ref{StripvsParabola}. The proof of (i) and (ii) is based in the ideas in the proof of \cite[Theorem 4.1]{Roi1} and follows similar steps. For any $\lambda\in i\mathbb{R}-c$ and $z\in i\mathbb{R}-r$, we have
\begin{eqnarray}\nonumber
G_{\lambda,z}&=&(-iA-\lambda)^{-1}(iA-z)^{-1}(B+\lambda)^{-1}(B+z)^{-1}\\\nonumber
&=&\frac{1}{z+\lambda}\big((-iA-\lambda)^{-1}+(iA-z)^{-1}\big)\frac{1}{z-\lambda}\big((B+z)^{-1}-(B+\lambda)^{-1}\big)\\\nonumber
&=&\frac{1}{\lambda^{2}-z^{2}}(iA+\lambda)^{-1}(B+z)^{-1}+\frac{1}{z^{2}-\lambda^{2}}(iA+\lambda)^{-1}(B+\lambda)^{-1}\\\label{Qexp}
&&+\frac{1}{\lambda^{2}-z^{2}}(-iA+z)^{-1}(B+z)^{-1}+\frac{1}{z^{2}-\lambda^{2}}(-iA+z)^{-1}(B+\lambda)^{-1}.
\end{eqnarray}

{\bf (i)} By \eqref{decayto0}, recall that
$$
H_{c}\subset \rho(\pm iA) \quad \text{and} \quad |\lambda|\|(\pm iA+\lambda)^{-1}\|_{\mathcal{L}(\mathcal{D}(A),X)}\leq K, \quad \lambda\in H_{c},
$$
for some $K>0$, where $H_{c}$ is defined in \eqref{StrVSParareas}. By Cauchy's theorem and \eqref{BDomdecay}, we have
\begin{eqnarray}\nonumber
\frac{i}{2}B^{-1}Av+\frac{1}{2}v&=&-\frac{1}{2\pi i}\Xint-_{i\mathbb{R}-r}\frac{1}{2z}(B+z)^{-1}(iA+B)vdz\\\nonumber
&&+ \frac{1}{2\pi i}\Xint-_{i\mathbb{R}-r}\frac{1}{2z}(iA+z)^{-1}(iA+B)vdz \\\nonumber
&=&\frac{1}{2\pi i}\Xint-_{i\mathbb{R}-r}\frac{1}{2z}(iA+z)^{-1}(B-iA)(B+z)^{-1}(iA+B)vdz\\\label{trm1}
&=&\frac{1}{2\pi i}\Xint-_{i\mathbb{R}-r}\frac{1}{2z}(iA+z)^{-1}(B+z)^{-1}(A^{2}+B^{2})vdz\\\nonumber
&=&\frac{1}{(2\pi i)^{2}}\Xint-_{i\mathbb{R}-r}\Xint-_{i\mathbb{R}-c}\frac{1}{\lambda^{2}-z^{2}}(iA+\lambda)^{-1}(B+z)^{-1}(A^{2}+B^{2})vd\lambda dz,
\end{eqnarray}
\begin{eqnarray}\nonumber
0&=&\frac{1}{2\pi i}\Xint-_{i\mathbb{R}-r}\frac{1}{2z}(B-z)^{-1}(iA+B)vdz\\\nonumber
&&-\frac{1}{2\pi i}\Xint-_{i\mathbb{R}-r}\frac{1}{2z}(iA+z)^{-1}(A+c)^{-1}(A+c)(iA+B)vdz\\\nonumber
&=&-\frac{1}{(2\pi i)^{2}}\Xint-_{i\mathbb{R}-r}\Xint-_{i\mathbb{R}-c}\frac{1}{z^{2}-\lambda^{2}}(B+\lambda)^{-1}(iA+B)vfd\lambda dz\\\nonumber
&&+\frac{1}{(2\pi i)^{2}}\Xint-_{i\mathbb{R}-r}\Xint-_{i\mathbb{R}-c}\frac{1}{z^{2}-\lambda^{2}}(iA+\lambda)^{-1}(iA+B)vd\lambda dz\\\nonumber
&=&\frac{1}{(2\pi i)^{2}}\Xint-_{i\mathbb{R}-r}\Xint-_{i\mathbb{R}-c}\frac{1}{z^{2}-\lambda^{2}}(iA+\lambda)^{-1}(B-iA)(B+\lambda)^{-1}(iA+B)vd\lambda dz\\\label{trm2}
&=&\frac{1}{(2\pi i)^{2}}\Xint-_{i\mathbb{R}-r}\Xint-_{i\mathbb{R}-c}\frac{1}{z^{2}-\lambda^{2}}(iA+\lambda)^{-1}(B+\lambda)^{-1}(A^{2}+B^{2})vd\lambda dz,
\end{eqnarray}
\begin{eqnarray}\nonumber
-\frac{i}{2}B^{-1}Av+\frac{1}{2}v&=&-\frac{1}{2\pi i}\Xint-_{i\mathbb{R}-r}\frac{1}{2z}(B+z)^{-1}(-iA+B)vdz\\\nonumber
&&+\frac{1}{2\pi i}\Xint-_{i\mathbb{R}-r}\frac{1}{2z}(-iA+z)^{-1}(-iA+B)vdz\\\nonumber
&=&\frac{1}{(2\pi i)^{2}}\Xint-_{i\mathbb{R}-r}\Xint-_{i\mathbb{R}-c}\frac{1}{z^{2}-\lambda^{2}}(B+z)^{-1}(-iA+B)vd\lambda dz\\\nonumber
&&-\frac{1}{(2\pi i)^{2}}\Xint-_{i\mathbb{R}-r}\Xint-_{i\mathbb{R}-c}\frac{1}{z^{2}-\lambda^{2}}(-iA+z)^{-1}(-iA+B)vd\lambda dz\\\nonumber
&=&\frac{1}{(2\pi i)^{2}}\Xint-_{i\mathbb{R}-r}\Xint-_{i\mathbb{R}-c}\frac{1}{\lambda^{2}-z^{2}}(-iA+z)^{-1} (iA+B)(B+z)^{-1}(-iA+B)vd\lambda dz\\\label{trm3}
&=&\frac{1}{(2\pi i)^{2}}\Xint-_{i\mathbb{R}-r}\Xint-_{i\mathbb{R}-c}\frac{1}{\lambda^{2}-z^{2}}(-iA+z)^{-1}(B+z)^{-1}(A^{2}+B^{2})vd\lambda dz,
\end{eqnarray}
and
\begin{eqnarray}\nonumber
0&=&\frac{1}{2\pi i}\Xint-_{i\mathbb{R}-r}\frac{1}{2z}(iA-z)^{-1}(B-z)^{-1}(A^{2}+B^{2})vdz\\\label{trm4}
&=&\frac{1}{(2\pi i)^{2}}\Xint-_{i\mathbb{R}-r}\Xint-_{i\mathbb{R}-c}\frac{1}{z^{2}-\lambda^{2}}(-iA+z)^{-1}(B+\lambda)^{-1}(A^{2}+B^{2})vd\lambda dz,
\end{eqnarray}
where in the first equality of the last equation we have also used Lemma \ref{decayBres}. Then, \eqref{invAB} follows from \eqref{trm1}-\eqref{trm4}, taking into account \eqref{Qexp}.

{\bf (ii)} By \eqref{intemb}, Remark \ref{StripvsParabola}, and \eqref{inrtSob0}, we write
$$
(A^{2}+\lambda^{2})^{-1}(B+\lambda)^{-1}w=\frac{1}{2\lambda}\big((iA+\lambda)^{-1}+(-iA+\lambda)^{-1}\big)(B+\lambda)^{-1}B^{-\nu+\varepsilon}B^{\nu-\varepsilon}w
$$
or
$$
(A^{2}+\lambda^{2})^{-1}(B+\lambda)^{-1}w=\frac{1}{2\lambda}\big((iA+\lambda)^{-1}+(-iA+\lambda)^{-1}\big)A^{-\nu+\varepsilon}(B+\lambda)^{-1}A^{\nu-\varepsilon}w,
$$
$\lambda\in i\mathbb{R}-c$, for all $\varepsilon>0$ sufficiently small. Hence, by Lemma \ref{fracpowdec}, \eqref{BONZ}, and \eqref{BRND}, the integral \eqref{intprop} converges absolutely, and \eqref{Lmapbounded} holds. Furthermore, by Lemma \ref{fracpowdec}, \eqref{BONZ}, \eqref{BRND}, and Cauchy's theorem, we have
\begin{eqnarray}\nonumber
\lefteqn{\frac{1}{2\pi i}\int_{i\mathbb{R}-c}(A^{2}+\lambda^{2})^{-1}(B+\lambda)^{-1}w d\lambda}\\\nonumber
&=&\frac{1}{2\pi i}\int_{i\mathbb{R}-r}\frac{1}{2z}\big((iA+z)^{-1}+(-iA+z)^{-1}\big)(B+z)^{-1}w dz\\\nonumber
&=&\frac{1}{2\pi i}\int_{i\mathbb{R}-r}\frac{1}{2z}(iA+z)^{-1}(B+z)^{-1}w dz\\\nonumber
&&+\frac{1}{2\pi i}\int_{i\mathbb{R}-r}\frac{1}{2z}(-iA+z)^{-1}(B+z)^{-1}w dz\\\nonumber
&&-\frac{1}{2\pi i}\int_{i\mathbb{R}-r}\frac{1}{2z}(-iA+z)^{-1}(B-z)^{-1}w dz\\\nonumber
&=&\frac{1}{(2\pi i)^{2}}\int_{i\mathbb{R}-r}\int_{i\mathbb{R}-c}\frac{1}{\lambda^{2}-z^{2}}(iA+\lambda)^{-1}(B+z)^{-1}w d\lambda dz\\\nonumber
&&+\frac{1}{(2\pi i)^{2}}\int_{i\mathbb{R}-r}\int_{i\mathbb{R}-c}\frac{1}{\lambda^{2}-z^{2}}(-iA+z)^{-1}(B+z)^{-1} w d\lambda dz\\\label{threeterms}
&&+\frac{1}{(2\pi i)^{2}}\int_{i\mathbb{R}-r}\int_{i\mathbb{R}-c}\frac{1}{z^{2}-\lambda^{2}}(-iA+z)^{-1}(B+\lambda)^{-1}w d\lambda dz.
\end{eqnarray}
On the other hand, by Lemma \ref{fracpowdec}, \eqref{BONZ}, \eqref{BRND}, and the dominated convergence theorem, the map 
$$
H_{r} \ni z\mapsto I_{\pm}(z)=\int_{i\mathbb{R}-c}\frac{1}{z\pm \lambda}(iA+\lambda)^{-1}(B+\lambda)^{-1}gd\lambda \in L^{p}(0,T;X)
$$
is analytic and satisfies
$$
I_{\pm}(z)\rightarrow 0 \quad \text{as} \quad |z|\rightarrow\infty, \quad z\in H_{r}.
$$
Hence, by the identity
$$
\int_{i\mathbb{R}-c}\frac{1}{z^{2}-\lambda^{2}}(iA+\lambda)^{-1}(B+\lambda)^{-1}gd\lambda=\frac{1}{2z}(I_{+}(z)+I_{-}(z)), \quad z\in H_{r},
$$
and Cauchy's theorem, we obtain
\begin{equation}\label{fourthterm}
\frac{1}{(2\pi i)^{2}}\Xint-_{i\mathbb{R}-r}\int_{i\mathbb{R}-c}\frac{1}{z^{2}-\lambda^{2}}(iA+\lambda)^{-1}(B+\lambda)^{-1}gd\lambda dz=0.
\end{equation}
Then, \eqref{inteq} follows by \eqref{Qexp}, \eqref{threeterms}, and \eqref{fourthterm}.

{\bf (iii)} If (a) holds, then $A^{2}f, B^{2}f\in W_{0}^{1,p}(0,T;\mathcal{D}(A))$. If (b) holds, then, by \eqref{BYfracemb} and complex interpolation, see, e.g., \cite[Theorem 2.6]{Lunardi}, we have that $A^{2}f, B^{2}f\in W_{0}^{s,p}(0,T;[X,\mathcal{D}(A)]_{\ell})$. Hence, in both cases, by (ii), the integrals in $LA^{2}f$, $LB^{2}f$ converge absolutely. Therefore, $Lf \in \mathcal{D}(A^{2}) \cap \mathcal{D}(B^{2})$
and 
\begin{equation}\label{AB2in}
A^{2}Lf=LA^{2}f, \quad B^{2}Lf= LB^{2}f.
\end{equation}
Hence, by (i), (ii), and Lemma \ref{fracderiv}, we obtain
\begin{eqnarray*}
\frac{1}{2\pi i}(A^{2}+B^{2})Lf&=&\frac{1}{2\pi i} L(A^{2}+B^{2})f\\
&=&\frac{1}{(2\pi i)^{2}}\Xint-_{i\mathbb{R}-r}\int_{i\mathbb{R}-c}G_{\lambda,z}(A^{2}+B^{2})fd\lambda dz=f. 
\end{eqnarray*}
The expression \eqref{expsolAWE} follows from \eqref{expresdt}. Uniqueness of solution follows by (i) and Lemma \ref{fracderiv}. Concerning the continuously dependence of the solution, note first that if (b) holds, then, by \cite[Corollary 15.3.10]{HNVW1}, up to an equivalent norm, we have
\begin{equation}\label{intopowers}
[\mathcal{D}(A^{2}),\mathcal{D}(A^{3})]_{\ell}=\mathcal{D}(A^{2+\ell}).
\end{equation}
Then, we write \eqref{AB2in} as
$$
A^{2}Lf=LA^{-1}B^{-1}A^{3}Bf, \quad B^{2}Lf= LA^{-1}B^{-1}AB^{3}f,
$$
in the case of (a), and, in the case of (b) as
$$
A^{2}Lf=LA^{-\ell}B^{-s+\varepsilon} A^{2+\ell}B^{s-\varepsilon}f, \quad B^{2}Lf= LA^{-\ell}B^{-s+\varepsilon}A^{\ell}B^{2+s-\varepsilon}f,
$$
due to \eqref{BYfracemb}, \eqref{intopowers}, and \cite[Theorem 15.3.9]{HNVW1}, where $\varepsilon>0$ is sufficiently small. Thus
$$
\|A^{2}u\|_{L^{p}(0,T;X)}+\|B^{2}u\|_{L^{p}(0,T;X)}\leq C_{0}(\|A^{3}Bf\|_{L^{p}(0,T;X)}+\|AB^{3}f\|_{L^{p}(0,T;X)})
$$
in the case of (a) and
$$
\|A^{2}u\|_{L^{p}(0,T;X)}+\|B^{2}u\|_{L^{p}(0,T;X)} \leq C_{1}( \|A^{2+\ell}B^{s-\varepsilon}f\|_{L^{p}(0,T;X)}+\|A^{\ell}B^{2+s-\varepsilon}f\|_{L^{p}(0,T;X)} )
$$
in the case of (b), for all $\varepsilon>0$ sufficiently small, with constants
$$
C_{0}=\|LA^{-1}B^{-1}\|_{\mathcal{L}(L^{p}(0,T;X))} \quad \text{and} \quad C_{1}=\|LA^{-\ell}B^{-s+\varepsilon}\|_{\mathcal{L}(L^{p}(0,T;X))},
$$
due to \eqref{intemb}, \eqref{BYfracemb}, and \eqref{Lmapbounded}. The result then follows by \eqref{BYfracemb}, \eqref{intopowers}, and \cite[Theorem 15.3.9]{HNVW1}.
\end{proof}

\begin{remark}\label{convprop}
In the above theorem, we can express all the relevant data in terms of the original operator $A^{2}=\Lambda$ instead of $A$, as follows.\\
{\bf (i)} Due to \cite[Theorem 15.2.7]{HNVW1}, if $\Lambda\in \mathcal{P}(0)$, then $A\in \mathcal{P}(0)$; and, if, in addition, $\Lambda\in \mathcal{BIP}$, then $A\in \mathcal{BIP}$ as well. In this case, the operator $\Lambda^{\frac{1}{2}}$, defined in the sense of fractional powers of sectorial operators, coincides with $A$. Moreover, by Remark \ref{StripvsParabola}, if $\Lambda\in \mathcal{V}_{c}$, then $A\in \mathcal{Z}_{c}$. \\
{\bf (ii)} If $\Lambda\in \mathcal{P}(0)$, then, due to \cite[(I.2.5.2) and (I.2.9.6)]{Amann1}, we have
$$
[X,\mathcal{D}(\Lambda)]_{\frac{1}{2}+\varepsilon} \hookrightarrow \mathcal{D}(A) \hookrightarrow[X,\mathcal{D}(\Lambda)]_{\frac{1}{2}-\varepsilon}, \quad [X,\mathcal{D}(\Lambda)]_{\frac{\nu}{2}+\varepsilon} \hookrightarrow [X,\mathcal{D}(A)]_{\nu} \hookrightarrow [X,\mathcal{D}(\Lambda)]_{\frac{\nu}{2}-\varepsilon}
$$
and
$$
[X,\mathcal{D}(\Lambda^{2})]_{\frac{3}{4}+\varepsilon} \hookrightarrow \mathcal{D}(A^{3}) \hookrightarrow [X,\mathcal{D}(\Lambda^{2})]_{\frac{3}{4}-\varepsilon}
$$
for all $\varepsilon>0$ sufficiently small. If, in addition, $\Lambda \in \mathcal{BIP}$, then 
$$
\mathcal{D}(A) = [X,\mathcal{D}(\Lambda)]_{\frac{1}{2}}, \quad [X,\mathcal{D}(A)]_{\nu} = [X,\mathcal{D}(\Lambda)]_{\frac{\nu}{2}}, \quad \text{and} \quad \mathcal{D}(A^{3})=[X,\mathcal{D}(\Lambda^{2})]_{\frac{3}{4}},
$$
up to equivalent norms, due to \cite[Theorem 15.3.9]{HNVW1}. \\
{\bf (iii)} If $\Lambda\in\mathcal{P}(\frac{\pi}{2})$, then, by \cite[(I.2.5.2) and (I.2.9.6)]{Amann1}, \cite[Corollary 7.3 (c)]{Haase1}, and \cite[Lemma 3.6]{RS3}, we have
$$
[X,\mathcal{D}(\Lambda^{2})]_{\frac{2+\ell}{4}+\varepsilon} \hookrightarrow [\mathcal{D}(A^{2}),\mathcal{D}(A^{3})]_{\ell} \hookrightarrow [X,\mathcal{D}(\Lambda^{2})]_{\frac{2+\ell}{4}-\varepsilon}
$$
for all $\varepsilon>0$ sufficiently small. If, in addition, $\Lambda \in \mathcal{BIP}$, then
$$
[\mathcal{D}(A^{2}),\mathcal{D}(A^{3})]_{\ell} = [X,\mathcal{D}(\Lambda^{2})]_{\frac{2+\ell}{4}},
$$
with equivalent norms, due to \cite[Theorem 15.2.7 and Corollary 15.3.10]{HNVW1} and \cite[Lemma 3.6]{RS3}. \mbox{\ } \hfill $\Diamond$
\end{remark}

\section{On the closure of $\pm iA+B$}\label{Sec5}

Based on the discussion at the beginning of Section \ref{Sec2}, denote by $\overline{\pm iA+B}$ the closure of 
\begin{equation}\label{iAplusB}
\pm iA+B:W_{0}^{1,p}(0,T;X)\cap L^{p}(0,T;\mathcal{D}(A)) \rightarrow L^{p}(0,T;X).
\end{equation}
We start by investigating some mapping and inversion properties of an appropriate extension of the operator $J_{\pm}$ in \eqref{Jdef}. The main tool for obtaining our extension is the theory of operator-valued Fourier multipliers.

\begin{theorem}\label{LASEWP}
Let $p\in (1,\infty)$, $c,T>0$, $X$ be a UMD complex Banach space, $A\in \mathcal{P}(0)\cap \mathcal{RZ}_{c}$ in $X$, and let $B$ be the derivation operator defined in \eqref{Bdef}. Then the operator $J_{\pm}$ in \eqref{bop1} admits an extension to a bounded map
$$
\text{from} \quad W_{0}^{1,p}(0,T;X) \quad \text{to} \quad L^{p}(0,T;X) \quad \text{and} \quad \text{from} \quad L^{p}(0,T;\mathcal{D}(A)) \quad \text{to} \quad L^{p}(0,T;X);
$$
the two extensions coincide on $W_{0}^{1,p}(0,T;X)\cap L^{p}(0,T;\mathcal{D}(A))$, and we continue to denote them by $J_{\pm}$ on this space. Moreover, there exists a $C>0$ such that
\begin{equation}\label{Jbound}
\|J_{\pm}g\|_{ L^{p}(0,T;X)}\leq C(\|g\|_{ L^{p}(0,T;X)}+\|(\mp iA+B)g\|_{ L^{p}(0,T;X)}\|)
\end{equation}
for any $g\in W_{0}^{3,p}(0,T;\mathcal{D}(A^{3}))$. In particular, $J_{\pm}$ extends further to a bounded map from $\mathcal{D}(\overline{\mp iA+B})$ to $L^{p}(0,T;X)$; we continue to denote this extension by $J_{\pm}$. If, in addition, $A\in \mathcal{P}(\frac{\pi}{2})$, then, for any $f \in \mathcal{D}(\overline{\mp iA+B})$, we have $J_{\pm}f \in \mathcal{D}(\overline{\pm iA+B})$ and
\begin{equation}\label{ElRinvJ}
\overline{\pm iA+B}J_{\pm}f=f.
\end{equation}
Furthermore,
\begin{equation}\label{linvcl}
J_{\pm}\overline{\pm iA+B}w=w \quad \text{for all} \quad w\in \mathcal{D}(\overline{\pm iA+B}) \quad \text{such that} \quad \overline{\pm iA+B}w \in \mathcal{D}(\overline{\mp iA+B}).
\end{equation}
\end{theorem}
\begin{proof}
Let $\xi\in W_{0}^{1,p}(0,T;[X,\mathcal{D}(A)]_{\ell})$ for some $\ell\in(0,1)$. Then, by \eqref{intemb}, \eqref{resdecDA}, and \eqref{BONZ}, for any $\theta\in (0,\ell)$, we have
\begin{eqnarray}\nonumber
\lefteqn{-\int_{i\mathbb{R}-c}\frac{1}{\lambda}(\pm iA-\lambda)^{-1}A^{-\theta}(B+\lambda)^{-1}BA^{\theta}\xi d\lambda}\\\nonumber
&=&\Xint-_{i\mathbb{R}-c}\frac{1}{\lambda}(\pm iA-\lambda)^{-1}A^{-\theta}A^{\theta}\xi d\lambda-\int_{i\mathbb{R}-c}\frac{1}{\lambda}(\pm iA-\lambda)^{-1}A^{-\theta}(B+\lambda)^{-1}BA^{\theta}\xi d\lambda\\\nonumber
&=&\int_{i\mathbb{R}-c}(\pm iA-\lambda)^{-1}A^{-\theta}\frac{1}{\lambda}(B^{-1}-(B+\lambda)^{-1})BA^{\theta}\xi d\lambda\\\label{JexpD}
&=&\int_{i\mathbb{R}-c}(\pm iA-\lambda)^{-1}A^{-\theta}(B+\lambda)^{-1}B^{-1}BA^{\theta}\xi d\lambda.
\end{eqnarray}
Moreover, let $r>2c$ and consider the bounded operator
$$
L^{p}(0,T;X)\ni v \mapsto K_{r}v=-\frac{1}{2\pi i}\int_{-c-ir}^{-c+ir}\frac{1}{\lambda}(B+\lambda)^{-1}vd\lambda \in L^{p}(0,T;X).
$$
By Cauchy's theorem, we have $K_{r}=B^{-1}+W_{r}$ with 
$$
W_{r}=\frac{1}{2\pi i}\int_{L_{r}}\frac{1}{\lambda}(B+\lambda)^{-1} d\lambda \in \mathcal{L}(L^{p}(0,T;X)),
$$
where $L_{r}=\{-c+re^{i\phi}\, |\, -\pi/2\leq\phi\leq \pi/2\}$ is oriented counterclockwise. By Lemma \ref{decayBres}, we get
\begin{equation}\label{Wrto0}
\lim_{r\rightarrow\infty}W_{r}v=0 \quad \text{for all} \quad v\in L^{p}(0,T;X).
\end{equation}
If $\eta\in W_{0}^{\ell,p}(0,T;\mathcal{D}(A))$, then by Lemma \ref{fracpowdec}, \eqref{resdecDA}, \eqref{BYfracemb}, and \eqref{Wrto0}, for any $\theta\in (0,\ell)$, we have
\begin{eqnarray*}
\lefteqn{2\pi iB^{-1}\eta \pm i\int_{i\mathbb{R}-c}\frac{1}{\lambda}(\pm iA-\lambda)^{-1}(B+\lambda)^{-1}B^{-\theta}B^{\theta}A \eta d\lambda}\\
&=&2\pi iB^{-1}\eta+2\pi i\lim_{r\rightarrow\infty}W_{r}\eta\pm i\int_{i\mathbb{R}-c}\frac{1}{\lambda}(\pm iA-\lambda)^{-1}(B+\lambda)^{-1}B^{-\theta}B^{\theta}A\eta d\lambda\\
&=&-\Xint-_{i\mathbb{R}-c}\frac{1}{\lambda}(B+\lambda)^{-1}B^{-\theta}B^{\theta}\eta d\lambda \pm i\int_{i\mathbb{R}-c}\frac{1}{\lambda}(\pm iA-\lambda)^{-1}(B+\lambda)^{-1}B^{-\theta}B^{\theta}A\eta d\lambda\\
&=&\int_{i\mathbb{R}-c}\frac{1}{\lambda}(-I\pm iA(\pm iA-\lambda)^{-1})A^{-1}(B+\lambda)^{-1}B^{-\theta}B^{\theta}A\eta d\lambda\\
&=&\int_{i\mathbb{R}-c}(\pm iA-\lambda)^{-1}A^{-1}(B+\lambda)^{-1}B^{-\theta}B^{\theta}A\eta d\lambda.
\end{eqnarray*}
Combined with \eqref{JexpD}, we conclude that
\begin{equation}\label{Jexp}
J_{\pm}u=\frac{1-(-1)^{j}}{2}B^{-1}u+\frac{1}{2\pi i}\int_{i\mathbb{R}-c}\frac{1}{\lambda}(\pm iA-\lambda)^{-1}(B+\lambda)^{-1}D_{j}ud\lambda, \quad u\in F_{j}, \quad j\in\{0,1\},
\end{equation}
where $D_{0}=-B$, $D_{1}=\pm iA$, $F_{0}=W_{0}^{1,p}(0,T;[X,\mathcal{D}(A)]_{\ell})$, and $F_{1}=W_{0}^{\ell,p}(0,T;\mathcal{D}(A))$, $\ell\in (0,1)$.

For any $\lambda\in i\mathbb{R}-c$ and $j\in\{0,1\}$, extend the functions
$$
D_{j}u \quad \text{and} \quad [0,T]\ni t\mapsto ((B+\lambda)^{-1}D_{j}u)(t)\in X, \quad \text{for} \quad u\in F_{j},
$$
outside $[0,T]$ by zero, and denote by $\mathfrak{L}(\cdot)$ the vector-valued Laplace transform, see, e.g., \cite[Section 1.4]{ABHN}. We have that $\mathfrak{L}(D_{j}u)(s)$ and $\mathfrak{L}((B+\lambda)^{-1}D_{j}u)(s)$ exist for all $s\in \mathbb{C}$. In particular, by \eqref{expresdt} and \cite[Proposition 1.6.4]{ABHN}, we have
\begin{equation}\label{LtrBl}
\mathfrak{L}((B+\lambda)^{-1}D_{j}u)(s)=\frac{1}{s+\lambda}\mathfrak{L}(D_{j}u)(s) \quad \text{for all} \quad s\in \mathbb{C} \quad \text{with} \quad \mathrm{Re}(s)>c.
\end{equation}
Note that for any $\zeta \in\mathbb{R}$, the map
\begin{equation}\label{LapPf}
\{z\in\mathbb{C}\, |\, \mathrm{Re}(z)=\zeta\}\ni s\mapsto \mathfrak{L}(D_{j}u)(s)=\int_{0}^{T}e^{-sx}(D_{j}u)(x)dx\in X
\end{equation}
is bounded. For any $\gamma>c$ and $t\in [0,T]$, let
$$
\mu(r)=\int_{\gamma-ir}^{\gamma+ir}e^{st}\mathfrak{L}((B+\lambda)^{-1}D_{j}u)(s)ds, \quad r>0.
$$
By \eqref{LtrBl}, for any $\rho>r>0$, we have 
\begin{eqnarray*}
\lefteqn{\mu(\rho)-\mu(r)=\int_{\gamma-i\rho}^{\gamma-ir}\frac{e^{st}}{s+\lambda}\mathfrak{L}(D_{j}u)(s)ds+\int_{\gamma+ir}^{\gamma+i\rho}\frac{e^{st}}{s+\lambda}\mathfrak{L}(D_{j}u)(s)ds}\\
&=&\int_{-\frac{\rho}{r}}^{-1}\frac{ire^{(\gamma+irs)t}}{\gamma+irs+\lambda}\mathfrak{L}(D_{j}u)(\gamma+irs)ds+\int_{1}^{\frac{\rho}{r}}\frac{ire^{(\gamma+irs)t}}{\gamma+irs+\lambda}\mathfrak{L}(D_{j}u)(\gamma+irs)ds.
\end{eqnarray*}
Hence, by \eqref{LapPf}, we obtain
$$
\lim_{r,\rho\rightarrow\infty, \, \rho/r\rightarrow 1}\|\mu(\rho)-\mu(r)\|_{X}=0,
$$
i.e., $\mu$ is {\em feebly oscillating} in the sense of \cite[Section 4.2.B]{ABHN}. Therefore, by \cite[Theorem 4.2.5 and Theorem 4.2.21]{ABHN} and \eqref{LtrBl}, we deduce that
\begin{equation} 
((B+\lambda)^{-1}D_{j}u)(t)=\frac{1}{2\pi i}\Xint-_{i\mathbb{R}+\gamma}\frac{e^{st}}{s+\lambda}\mathfrak{L}(D_{j}u)(s)ds
\end{equation}
uniformly in $t\in [0,T]$. As a consequence \eqref{Jexp}, becomes
$$
(J_{\pm}u)(t)=\frac{1-(-1)^{j}}{2}(B^{-1}u)(t)\pm \frac{i}{(2\pi i)^{2}}\int_{i\mathbb{R}+c}\frac{1}{\lambda}(A\mp i\lambda)^{-1}\Big(\Xint-_{i\mathbb{R}+\gamma}\frac{e^{st}}{s-\lambda}\mathfrak{L}(D_{j}u)(s)ds\Big)d\lambda,
$$
$t\in [0,T]$. Hence, if we restrict to $g\in W_{0}^{3,p}(0,T;\mathcal{D}(A^{3}))$ and extend $D_{j}g$, $j\in\{0,1\}$, outside $[0,T]$ by zero, we get 
\begin{equation}\label{Jresttog}
(J_{\pm}g)(t)=\frac{1-(-1)^{j}}{2}(B^{-1}g)(t)\pm\frac{i}{(2\pi i)^{2}}\int_{i\mathbb{R}+c}\frac{1}{\lambda}(A\mp i\lambda)^{-1}A^{-1}\Big(\Xint-_{i\mathbb{R}+\gamma}\frac{e^{st}}{s-\lambda}\mathfrak{L}(AD_{j}g)(s)ds\Big)d\lambda,
\end{equation}
$t\in [0,T]$. Integration by parts in \eqref{LapPf} yields
$$
\mathfrak{L}(AD_{j}g)(s)=-\frac{e^{-sT}}{s}(AD_{j}g)(T)+\frac{1}{s}\int_{0}^{T}e^{-sx}(BAD_{j}g)(x)dx, \quad s\in i\mathbb{R}+\gamma.
$$
Note that by Young’s inequality for convolution, the second integral in \eqref{Jresttog} is uniformly bounded in $\lambda$. By the above equation, \eqref{resdecDA}, and Fubini's theorem (see, e.g., \cite[Theorem 1.1.9]{ABHN}), from \eqref{Jresttog} we obtain
\begin{eqnarray*}
(J_{\pm}g)(t)&=&\frac{1-(-1)^{j}}{2}(B^{-1}g)(t)\\
&&\mp\frac{ i}{(2\pi i)^{2}}\int_{i\mathbb{R}+\gamma}\Big(\int_{i\mathbb{R}+c}\frac{1}{\lambda(\lambda-s)}(A\mp i\lambda)^{-1}A^{-1}d\lambda\Big)e^{st}\mathfrak{L}(AD_{j}g)(s)ds,
\end{eqnarray*}
$t\in [0,T]$, and, by Cauchy's theorem, we deduce
$$
(J_{\pm}g)(t)=\frac{1-(-1)^{j}}{2}(B^{-1}g)(t) \pm \frac{1}{2\pi}\int_{i\mathbb{R}+\gamma}\frac{e^{st}}{s}(A\mp is)^{-1}\mathfrak{L}(D_{j}g)(s)ds,
$$
$t\in [0,T]$. Therefore,
\begin{equation}\label{FourJexp}
(J_{\pm}g)(t)=\frac{1-(-1)^{j}}{2}(B^{-1}g)(t)+\frac{e^{\gamma t}}{2\pi}\int_{\mathbb{R}}\frac{e^{ist}}{\pm s\mp i\gamma}(A\pm s\mp i\gamma)^{-1}\Big(\int_{\mathbb{R}}e^{-isx}e^{-\gamma x}(D_{j}g)(x)dx\Big)ds,
\end{equation}
$t\in [0,T]$. 

At this point, we apply the Mihlin-type Fourier multiplier theorem for operator-valued multiplier functions \cite[Theorem 3.4]{Weis}, due to Weis. The map
$$
\mathbb{R}\ni s\mapsto M(s)=\frac{1}{\pm s\mp i\gamma}(A\pm s\mp i\gamma)^{-1}\in \mathcal{L}(X)
$$
is differentiable with 
$$
M'(s)=\mp\frac{1}{(\pm s\mp i\gamma)^{2}}(A\pm s\mp i\gamma)^{-1} \mp\frac{1}{\pm s\mp i\gamma}(A\pm s\mp i\gamma)^{-2}, \quad s\in\mathbb{R}.
$$
Hence, by the assumption and Kahane’s contraction principle (see, e.g., \cite[Proposition 2.5]{KuWe}), we see that both
$$
\mathbb{R}\ni s\mapsto M(s) \in \mathcal{L}(X) \quad \text{and} \quad \mathbb{R}\ni s\mapsto sM'(s) \in \mathcal{L}(X)
$$
are $R$-bounded. Thus, by \eqref{FourJexp} and \cite[Theorem 3.4]{Weis}, there exists a $C_{0}>0$ such that
$$
\|J_{\pm}g\|_{L^{p}(0,T;X)}\leq C_{0}\|D_{j}g\|_{L^{p}(0,T;X)}.
$$
Therefore, by the density of $W_{0}^{3,p}(0,T;\mathcal{D}(A^{3}))$ in each of the spaces $W_{0}^{1,p}(0,T;X)$ and $L^{p}(0,T;\mathcal{D}(A))$, we infer that 
$$
J_{\pm}\in \mathcal{L}(W_{0}^{1,p}(0,T;X),L^{p}(0,T;X)) \quad \text{and} \quad J_{\pm}\in \mathcal{L}(L^{p}(0,T;\mathcal{D}(A)),L^{p}(0,T;X))
$$ 
in the sense of a unique bounded extension. Denote these two extensions, respectively, by $J_{\pm,0}$ and $J_{\pm,1}$. For any $\beta\in W_{0}^{1,p}(0,T;X)\cap L^{p}(0,T;\mathcal{D}(A))$, let $g_{k}\in W_{0}^{3,p}(0,T;\mathcal{D}(A^{3}))$, $k\in\mathbb{N}$, be such that $g_{k}\rightarrow \beta$ as $k\rightarrow \infty$ in the $W_{0}^{1,p}(0,T;X)\cap L^{p}(0,T;\mathcal{D}(A))$-topology (i.e., the maximum of the two norms). By
$$
\|J_{\pm,0}\beta-J_{\pm,1}\beta\|_{L^{p}(0,T;X)}\leq \|J_{\pm,0}\beta-J_{\pm}g_{k}\|_{L^{p}(0,T;X)}+\|J_{\pm,1}\beta-J_{\pm}g_{k}\|_{L^{p}(0,T;X)},
$$
we see that $J_{\pm,0}\beta=J_{\pm,1}\beta$. Furthermore, summing \eqref{FourJexp} for $j=0$ and $j=1$, we get 
\begin{eqnarray}\label{JestBMDomtoLp}
\lefteqn{(2J_{\pm}g)(t)}\\\nonumber
&=&(B^{-1}g)(t)+\frac{e^{\gamma t}}{2\pi}\int_{\mathbb{R}}\frac{e^{ist}}{\pm s\mp i\gamma}(A\pm s\mp i\gamma)^{-1}\Big(\int_{\mathbb{R}}e^{-isx}e^{-\gamma x}((\pm iA-B)g)(x)dx\Big)ds,
\end{eqnarray}
$t\in [0,T]$. Hence, \eqref{Jbound} follows by \cite[Theorem 3.4]{Weis}. 

To extend $J_{\pm}$ further, let $\nu\in \mathcal{D}(\overline{\mp iA+B})$ and $\nu_{k}\in W_{0}^{3,p}(0,T;\mathcal{D}(A^{3}))$, $k\in\mathbb{N}$, such that $\nu_{k}\rightarrow \nu$ and $(\mp iA+B)\nu_{k}\rightarrow \overline{\mp iA+B} \nu$ in $L^{p}(0,T;X)$ as $k\rightarrow \infty$. By \eqref{Jbound}, we have that $\{J_{\pm}\nu_{k}\}_{k\in\mathbb{N}}$ is Cauchy in $L^{p}(0,T;X)$, and hence converges to some $y\in L^{p}(0,T;X)$. Define $J_{\pm} \nu=y$. Let $\phi_{k}\in W_{0}^{3,p}(0,T;\mathcal{D}(A^{3}))$, $k\in\mathbb{N}$, such that $\phi_{k}\rightarrow \nu$ and $(\mp iA+B)\phi_{k}\rightarrow \overline{\mp iA+B}\nu$ in $L^{p}(0,T;X)$ as $k\rightarrow \infty$. Denote by $J_{\pm}\nu$ and $J'_{\pm}\nu$ the limits of $\{J_{\pm}\nu_{k}\}_{k\in\mathbb{N}}$ and $\{J_{\pm}\phi_{k}\}_{k\in\mathbb{N}}$, respectively. Then \eqref{Jbound} implies
$$ 
\|J_{\pm}(\nu_{k}-\phi_{k})\|_{ L^{p}(0,T;X)} \leq C(\|\nu_{k}-\phi_{k}\|_{ L^{p}(0,T;X)}+\|(\mp iA+B)\nu_{k}-(\mp iA+B)\phi_{k}\|_{ L^{p}(0,T;X)}),
$$
$k\in\mathbb{N}$, so that, by the inequality
$$
 \|J_{\pm}\nu-J'_{\pm}\nu\|_{ L^{p}(0,T;X)} \leq \|J_{\pm}\nu-J_{\pm}\nu_{k}\|_{ L^{p}(0,T;X)} +\|J_{\pm}(\nu_{k}-\phi_{k})\|_{ L^{p}(0,T;X)} +\|J_{\pm}\phi_{k}-J'_{\pm}\nu\|_{ L^{p}(0,T;X)},
$$
we see that the above extension is well defined. Finally, by taking the limit in \eqref{Jbound}, we obtain $J_{\pm}\in\mathcal{L}(\mathcal{D}(\overline{\mp iA+B}), L^{p}(0,T;X))$. 
 
Now, take $f \in \mathcal{D}(\overline{\mp iA+B})$ and $v_{k}\in W_{0}^{3,p}(0,T;\mathcal{D}(A^{3}))$, $k\in\mathbb{N}$, such that $v_{k}\rightarrow f$ in $ \mathcal{D}(\overline{\mp iA+B})$ as $k\rightarrow \infty$. By the boundedness of $J_{\pm}$, we get $J_{\pm}v_{k}\rightarrow J_{\pm}f$ in $L^{p}(0,T;X)$ as $k\rightarrow \infty$. On the other hand, by Theorem \ref{aseth}(iii), we have $J_{\pm}v_{k}\in \mathcal{D}(\pm iA+B)$ and
$$
(\pm iA+B)J_{\pm}v_{k}=v_{k}\rightarrow f
$$
as $k\rightarrow \infty$, so that, by the closedness of $\overline{\pm iA+B}$ in $L^{p}(0,T;X)$, we get $J_{\pm}f \in \mathcal{D}(\overline{\pm iA+B})$ and \eqref{ElRinvJ} holds.

If we take the difference between \eqref{FourJexp} for $j=0$ and $j=1$, we get 
$$
-(B^{-1}g)(t)=\frac{e^{\gamma t}}{2\pi}\int_{\mathbb{R}}\frac{e^{ist}}{\pm s\mp i\gamma}(A\pm s\mp i\gamma)^{-1}\Big(\int_{\mathbb{R}}e^{-isx}e^{-\gamma x}((\pm iA+B)g)(x)dx\Big)ds,
$$
$t\in [0,T]$. Thus, by \cite[Theorem 3.4]{Weis}, there exists a $C_{1}>0$ such that
\begin{equation}\label{furthbnd}
\|B^{-1}g\|_{L^{p}(0,T;X)}\leq C_{1}\|(\pm iA+B)g\|_{L^{p}(0,T;X)}.
\end{equation}
Let $\omega\in \mathcal{D}(\overline{\pm iA+B})$ satisfy $\overline{\pm iA+B}\omega=0$. Take $\omega_{k}\in W_{0}^{3,p}(0,T;\mathcal{D}(A^{3}))$, $k\in\mathbb{N}$, such that $\omega_{k}\rightarrow \omega$ and $(\pm iA+B)\omega_{k}\rightarrow 0$ in $L^{p}(0,T;X)$ as $k\rightarrow \infty$. Then, \eqref{furthbnd} implies $B^{-1}\omega_{k}\rightarrow 0$ as $k\rightarrow \infty$, i.e., $\omega=0$ and $\overline{\pm iA+B}$ is injective. Then, setting $f= \overline{\pm iA+B}w$ in \eqref{ElRinvJ}, we obtain 
$$
\overline{\pm iA+B}(J_{\pm}\overline{\pm iA+B}w -w)=0,
$$
which shows \eqref{linvcl}.
\end{proof}

Based on the above theorem, we can now obtain well-posedness of certain extensions of \eqref{ACLSE} and \eqref{ACLWE}, which are defined through $\overline{\pm iA+B}$; see \eqref{stabawe} and \eqref{stabase} below. The forcing term $f$ is now permitted to belong to a space that contains the space of solutions.

\begin{theorem}[abstract wave and Schr\"odinger equations]\label{AWEWP}
Let $p\in (1,\infty)$, $c,T>0$, $X$ be a UMD complex Banach space, $A\in \mathcal{P}(0)\cap\mathcal{RZ}_{c}$ in $X$, $B$ be the derivation operator defined in \eqref{Bdef}, and $J_{\pm}$ be the operator from Theorem \ref{LASEWP}. Let the space
$$
E_{0}^{\pm}=\{v\in \mathcal{D}(\overline{\pm iA+B}) \, |\, \overline{\pm iA+B}v\in \mathcal{D}(\overline{\mp iA+B})\}
$$
be endowed with the norm
$$
v\mapsto \|v\|_{E_{0}^{\pm}}=\|v\|_{L^{p}(0,T;X)}+\|\overline{\pm iA+B}v\|_{L^{p}(0,T;X)}+\|\overline{\mp iA+B}\, \overline{\pm iA+B}v\|_{L^{p}(0,T;X)},
$$ 
and let
$$
E_{2}^{\pm}=\{v\in E_{0}^{\pm} \, |\, \overline{\mp iA+B} \, \overline{\pm iA+B}v\in E_{0}^{\pm}\}.
$$
Then $E_{0}^{\pm}$ is a Banach space, and the operator
\begin{equation}\label{ABPMOP}
\overline{\mp iA+B} \, \overline{\pm iA+B}: E_{2}^{\pm}\rightarrow E_{0}^{\pm}
\end{equation}
is closed and, moreover, invertible; its inverse is $J_{\pm}J_{\mp}\in\mathcal{L}(E_{0}^{\pm},E_{2}^{\pm})$. Furthermore:\\
{\em(abstract wave equation)} For any $f\in E_{0}^{\pm}$, there exists a unique $u\in E_{2}^{\pm}$ satisfying 
\begin{equation}\label{stabawe}
\overline{\mp iA+B} \, \overline{\pm iA+B}u=f;
\end{equation}
the solution is given by $u=J_{\pm}J_{\mp}f$ and depends continuously on $f$. In addition, $u$ is unique in $E_{0}^{\pm}$, even if $f\in L^{p}(0,T;X)$. \\
{\em (abstract Schr\"odinger equation)} Let the space
$$
E_{1}^{\pm}=\{v \in \mathcal{D}(\overline{\mp iA+B}) \, |\, \overline{\mp iA+B} v\in E_{0}^{\pm}\}
$$
be endowed with the norm 
$$
v\mapsto \|v\|_{E_{1}^{\pm}}=\|v\|_{L^{p}(0,T;X)}+\|\overline{\mp iA+B}v\|_{E_{0}^{\pm}}.
$$
For any $f\in E_{0}^{\pm}$, there exists a unique $w\in E_{1}^{\pm}$ satisfying 
\begin{equation}\label{stabase}
\overline{\mp iA+B}w=f ;
\end{equation}
the solution is given by $w=J_{\mp}f$ and depends continuously on $f$. In addition, $w$ is unique in $\mathcal{D}(\overline{\mp iA+B})$, even if $f\in L^{p}(0,T;X)$.
\end{theorem}
\begin{proof}
Let $\{g_{k}\}_{k\in\mathbb{N}}$ be a Cauchy sequence in $E_{0}^{\pm}$. Then $\{g_{k}\}_{k\in\mathbb{N}}$ is Cauchy in $L^{p}(0,T;X)$, and hence it converges to some $g\in L^{p}(0,T;X)$. Moreover, $\{\overline{\pm iA+B}g_{k}\}_{k\in\mathbb{N}}$ is also Cauchy in $L^{p}(0,T;X)$, and hence it converges to some $h\in L^{p}(0,T;X)$. By the closedness of $\overline{\pm iA+B}$ in $L^{p}(0,T;X)$, we have $g\in \mathcal{D}(\overline{\pm iA+B})$ and $\overline{\pm iA+B}g=h$. Furthermore, $\{\overline{\mp iA+B}\,\overline{\pm iA+B}g_{k}\}_{k\in\mathbb{N}}$ is Cauchy in $L^{p}(0,T;X)$, and thus it converges to some $\xi\in L^{p}(0,T;X)$. By the closedness of $\overline{\mp iA+B}$ in $L^{p}(0,T;X)$, we deduce that $h\in \mathcal{D}(\overline{\mp iA+B})$ and $\overline{\mp iA+B}h=\xi$ . This implies that $g\in E_{0}^{\pm}$ and
\begin{eqnarray*}
\lefteqn{\|g-g_{k}\|_{E_{0}^{\pm}}}\\
&=&\|g-g_{k}\|_{L^{p}(0,T;X)}+\|h-\overline{\pm iA+B}g_{k}\|_{L^{p}(0,T;X)}+\|\xi-\overline{\mp iA+B}\, \overline{\pm iA+B}g_{k}\|_{L^{p}(0,T;X)},
\end{eqnarray*}
$k\in\mathbb{N}$, i.e., $E_{0}^{\pm}$ is complete. 

For any $\eta\in E_{0}^{\pm}$, we have that $\overline{\mp iA+B} \, \overline{\pm iA+B}\eta$ is well-defined in $L^{p}(0,T;X)$, so that the operator in \eqref{ABPMOP} is also well-defined. Let $\phi\in E_{2}^{\pm}$. Since $\phi\in E_{0}^{\pm}$, we have that $ \overline{\pm iA+B}\phi\in \mathcal{D}(\overline{\mp iA+B} )$. Also, $\overline{\mp iA+B} \, \overline{\pm iA+B}\phi\in E_{0}^{\pm}$, and hence $\overline{\mp iA+B} \, \overline{\pm iA+B}\phi\in \mathcal{D}(\overline{\pm iA+B})$. Therefore, by \eqref{linvcl}, we have 
\begin{equation}\label{fstpinv}
J_{\mp}\overline{\mp iA+B} \, \overline{\pm iA+B}\phi= \overline{\pm iA+B}\phi.
\end{equation}
Again, by \eqref{linvcl}, we get 
$$
J_{\pm}J_{\mp}\overline{\mp iA+B} \, \overline{\pm iA+B}\phi= J_{\pm}\overline{\pm iA+B}\phi=\phi,
$$
which shows that $J_{\pm}J_{\mp}$ is a left inverse of the operator in \eqref{ABPMOP}. Also, by applying \eqref{ElRinvJ} twice, we have 
$$
\overline{\mp iA+B} \, \overline{\pm iA+B}J_{\pm}J_{\mp}f=\overline{\mp iA+B} J_{\mp}f= f \quad \text{and} \quad J_{\pm}J_{\mp}f\in E_{0}^{\pm}
$$
for any $f\in E_{0}^{\pm}$, i.e., $J_{\pm}J_{\mp}f\in E_{2}^{\pm}$, and $J_{\pm}J_{\mp}$ is also a right inverse of the operator in \eqref{ABPMOP}.

Teke $\nu,\varphi\in E_{0}^{\pm}$ and $\nu_{k}\in E_{2}^{\pm}$, $k\in\mathbb{N}$, such that $\nu_{k}\rightarrow \nu$ and $\overline{\mp iA+B} \, \overline{\pm iA+B}\nu_{k}\rightarrow \varphi$ in $E_{0}^{\pm}$ as $k\rightarrow \infty$. We have that both $\overline{\mp iA+B} \, \overline{\pm iA+B}\nu_{k}$, $k\in\mathbb{N}$, and $\varphi$ belong to $\mathcal{D}(\overline{\pm iA+B})$, and also $\overline{\mp iA+B} \, \overline{\pm iA+B}\nu_{k}\rightarrow \varphi$ in the $\mathcal{D}(\overline{\pm iA+B})$-topology, so that $J_{\mp}\overline{\mp iA+B} \, \overline{\pm iA+B}\nu_{k}\rightarrow J_{\mp}\varphi$ in $L^{p}(0,T;X)$ as $k\rightarrow \infty$ due to Theorem \ref{LASEWP}. By \eqref{fstpinv}, we obtain that $\overline{\pm iA+B}\nu_{k}\rightarrow J_{\mp}\varphi$ in $L^{p}(0,T;X)$ as $k\rightarrow \infty$. Hence:\\
(i) By the closedness of $\overline{\pm iA+B}$ in $L^{p}(0,T;X)$, we find that $\nu\in \mathcal{D}(\overline{\pm iA+B})$ and also $\overline{\pm iA+B}\nu=J_{\mp}\varphi$. \\
(ii) By the closedness of $\overline{\mp iA+B}$ in $L^{p}(0,T;X)$, we find that $J_{\mp}\varphi\in \mathcal{D}(\overline{\mp iA+B})$ and $\overline{\mp iA+B}J_{\mp}\varphi=\varphi$.\\
In particular, $\overline{\pm iA+B}\nu \in \mathcal{D}(\overline{\mp iA+B})$ and $\overline{\mp iA+B}\, \overline{\pm iA+B}\nu=\varphi$, i.e., $\nu\in E_{2}^{\pm}$ and the operator \eqref{ABPMOP} is closed. The fact that $J_{\pm}J_{\mp}\in\mathcal{L}(E_{0}^{\pm},E_{2}^{\pm})$ follows from the inverse mapping theorem, and the uniqueness of the solution of \eqref{stabawe} in $E_{0}^{\pm}$, even in the case $f\in L^{p}(0,T;X)$, follows by \eqref{linvcl}.

Due to \eqref{stabawe}, for any $f\in E_{0}^{\pm}$, there exists a solution $w=\overline{\pm iA+B}J_{\pm}J_{\mp}f$ of \eqref{stabase}. This solution belongs to 
\begin{eqnarray*}
\lefteqn{\{\overline{\pm iA+B}u \, | \, u\in E_{2}^{\pm}\}}\\ 
&=&\{\overline{\pm iA+B}u\in \mathcal{D}(\overline{\mp iA+B}) \, |\, u\in \mathcal{D}(\overline{\pm iA+B}) \,\, \text{and} \,\, \overline{\mp iA+B} \, \overline{\pm iA+B}u\in E_{0}^{\pm}\}\\
&=& E_{1}^{\pm},
\end{eqnarray*}
where the last equality follows from \eqref{ElRinvJ} and \eqref{linvcl}. The uniqueness of the solution of \eqref{stabase} in $\mathcal{D}(\overline{\mp iA+B})$, even for $f\in L^{p}(0,T;X)$, follows from \eqref{linvcl}. By \eqref{ElRinvJ}, we also have $w=J_{\mp}f$. Furthermore, \eqref{ElRinvJ} implies
\begin{eqnarray*}
\lefteqn{\|w\|_{E_{1}^{\pm}}=\|J_{\mp}f\|_{L^{p}(0,T;X)}+\|\overline{\mp iA+B}J_{\mp}f\|_{L^{p}(0,T;X)}}\\
&&+\|\overline{\pm iA+B}\, \overline{\mp iA+B}J_{\mp}f\|_{L^{p}(0,T;X)}+\|\overline{\mp iA+B} \, \overline{\pm iA+B}\, \overline{\mp iA+B}J_{\mp}f\|_{L^{p}(0,T;X)}\\
&=&\|J_{\mp}f\|_{L^{p}(0,T;X)}+\|f\|_{L^{p}(0,T;X)}+\|\overline{\pm iA+B}f\|_{L^{p}(0,T;X)}+\|\overline{\mp iA+B} \, \overline{\pm iA+B}f\|_{L^{p}(0,T;X)}\\
&\leq&C_{0}\big(\|f\|_{L^{p}(0,T;X)}+\|\overline{\pm iA+B}f\|_{L^{p}(0,T;X)}+\|\overline{\mp iA+B} \, \overline{\pm iA+B}f\|_{L^{p}(0,T;X)}\big)=C_{0}\|f\|_{E_{0}^{\pm}}
\end{eqnarray*}
for some $C_{0}>0$, where we have used Theorem \ref{LASEWP} in the inequality.
\end{proof}

In the above theorem, we obtained solutions of \eqref{stabase} through the solution operator of \eqref{stabawe}. However, based on Theorem \ref{LASEWP}, we can obtain solutions of \eqref{stabase} in a more direct way, as follows.

\begin{theorem}\label{BestScrth}
Let $p\in (1,\infty)$, $c,T>0$, $X$ be a UMD complex Banach space, $A\in \mathcal{P}(0)\cap\mathcal{RZ}_{c}$ in $X$, $B$ be the derivation operator defined in \eqref{Bdef}, and let $J_{\pm}$ be the operator from Theorem \ref{LASEWP}. Assume that 
\begin{equation}\label{DeqD}
\mathcal{D}(\overline{iA+B})=\mathcal{D}(\overline{-iA+B}) \quad \text{as sets,}
\end{equation}
and let $X_{0}^{\pm}=\mathcal{D}(\overline{\pm iA+B})$ and $X_{1}^{\pm}=\{u\in X_{0}^{\pm} \, |\, \overline{\pm iA+B}u\in X_{0}^{\pm}\}$. Then, the operator
\begin{equation}\label{astrop}
X_{1}^{\pm}\ni u\mapsto \overline{\pm iA+B}u\in X_{0}^{\pm} 
\end{equation}
is closed and invertible; its inverse is $J_{\pm}\in\mathcal{L}(X_{0}^{\pm},X_{1}^{\pm})$.
\end{theorem}
\begin{proof}
By \eqref{ElRinvJ}, \eqref{linvcl}, and \eqref{DeqD}, we have that $J_{\pm}$ is the inverse of the operator in \eqref{astrop}. Concerning the closedness of this operator, let $u_{k}\in X_{1}^{\pm}$, $k\in\mathbb{N}$, be such that $u_{k}\rightarrow u$ and $\overline{\pm iA+B}u_{k}\rightarrow w$ in $X_{0}^{\pm}$ as $k\rightarrow\infty$, for some $u,w\in X_{0}^{\pm}$. By the closedness of $\overline{\pm iA+B}$ in $L^{p}(0,T;X)$, we have that $\overline{\pm iA+B}u=w$, i.e., $u\in X_{1}^{\pm}$. The fact that $J_{\pm}\in \mathcal{L}(X_{0}^{\pm},X_{1}^{\pm})$ follows from the inverse mapping theorem.
\end{proof}

\section{An abstract semilinear wave equation}\label{Sec6}

In this section, we use the well-posedness of \eqref{stabawe} to study an abstract semilinear wave equation. We begin with the following uniform boundedness result. 

\begin{lemma}\label{CTindep}
Let $p\in (1,\infty)$, $c>0$, $0<T\leq T_{0}$, $X$ be a UMD complex Banach space, $A\in \mathcal{P}(0)\cap\mathcal{RZ}_{c}$ in $X$, and let $B$ be the derivation operator defined in \eqref{Bdef}. Then, there exists a $C>0$, independent of $T\in (0,T_{0}]$, such that
$$
\|u\|_{E_{2}^{\pm}}\leq C\|f\|_{E_{0}^{\pm}} \quad \text{for any} \quad f\in E_{0}^{\pm},
$$
where $u\in E_{2}^{\pm}$ is, according to Theorem \ref{AWEWP}, the unique solution of \eqref{stabawe}.
\end{lemma}
\begin{proof}
Since we do not have a description of the space $E_{0}^{\pm}$ in terms of vector-valued Sobolev spaces, and we have no information concerning the closedness of the sum $\pm iA+B$, we cannot follow the simple argument in the proof of \cite[Proposition 2.2]{CL}. Therefore, we will use the inverse $J_{\pm}J_{\mp}$ from Theorem \ref{AWEWP}. We estimate
\begin{eqnarray*}
\|J_{\pm}J_{\mp}\|_{\mathcal{L}(E_{0}^{\pm},E_{2}^{\pm})}&=&\sup_{v\in E_{0}^{\pm}, \|v\|_{E_{0}^{\pm}}=1 }\|J_{\pm}J_{\mp}v\|_{E_{2}^{\pm}}\\
&=&\sup_{v\in E_{0}^{\pm}, \|v\|_{E_{0}^{\pm}}=1 }( \|J_{\pm}J_{\mp}v\|_{E_{0}^{\pm}} + \|\overline{\mp iA+B} \, \overline{\pm iA+B}J_{\pm}J_{\mp}v\|_{E_{0}^{\pm}} )\\
&=&1+\sup_{v\in E_{0}^{\pm}, \|v\|_{E_{0}^{\pm}}=1 } \|J_{\pm}J_{\mp}v\|_{E_{0}^{\pm}},
\end{eqnarray*}
where we have used \eqref{ElRinvJ} in the last step. Thus, again by \eqref{ElRinvJ}, we get
\begin{eqnarray*}
\|J_{\pm}J_{\mp}\|_{\mathcal{L}(E_{0}^{\pm},E_{2}^{\pm})} &=&1+\sup_{v\in E_{0}^{\pm}, \|v\|_{E_{0}^{\pm}}=1 } (\|J_{\pm}J_{\mp}v \|_{L^{p}(0,T;X)} \\
&&+\|\overline{\pm iA+B}J_{\pm}J_{\mp}v \|_{L^{p}(0,T;X)}+\|\overline{\mp iA+B}\, \overline{\pm iA+B}J_{\pm}J_{\mp}v \|_{L^{p}(0,T;X)} )\\
&\leq&2+\sup_{v\in E_{0}^{\pm}, \|v\|_{E_{0}^{\pm}}=1 } (\|J_{\pm}J_{\mp}v \|_{L^{p}(0,T;X)} +\|J_{\mp}v \|_{L^{p}(0,T;X)})\\
&\leq&2+\sup_{v\in E_{0}^{\pm}, \|v\|_{E_{0}^{\pm}}=1 } (C_{\pm}\|J_{\mp}v \|_{\mathcal{D}(\overline{\mp iA+B})} +C_{\mp}\|v \|_{\mathcal{D}(\overline{\pm iA+B})})\\
&\leq&2+C_{0}\sup_{v\in E_{0}^{\pm}, \|v\|_{E_{0}^{\pm}}=1 } (\|J_{\mp}v \|_{\mathcal{D}(\overline{\mp iA+B})} +1),
\end{eqnarray*}
where $C_{\pm}=\|J_{\pm}\|_{\mathcal{L}( \mathcal{D}(\overline{\mp iA+B}), L^{p}(0,T;X))}$, $C_{0}=\max\{C_{-},C_{+}\}$, and we have used Theorem \ref{LASEWP}. Therefore, again by Theorem \ref{LASEWP} and \eqref{ElRinvJ}, we find
\begin{eqnarray*}
\|J_{\pm}J_{\mp}\|_{\mathcal{L}(E_{0}^{\pm},E_{2}^{\pm})} &\leq&2+C_{0}\sup_{v\in E_{0}^{\pm}, \|v\|_{E_{0}^{\pm}}=1 } (\|J_{\mp}v \|_{L^{p}(0,T;X)} +\|\overline{\mp iA+B}J_{\mp}v \|_{L^{p}(0,T;X)} +1)\\
&\leq&2+C_{0}\sup_{v\in E_{0}^{\pm}, \|v\|_{E_{0}^{\pm}}=1 } (C_{\mp}\|v \|_{\mathcal{D}(\overline{\pm iA+B})} +\|v \|_{L^{p}(0,T;X)} +1)\\
&\leq&2(1+C_{0}(C_{0}+1)).
\end{eqnarray*}
Hence, it suffices to show that the norms $C_{\pm}$ are uniformly bounded in $T\in (0,T_{0}]$. By \eqref{BRND}, \eqref{JestBMDomtoLp}, and \cite[Theorem 3.4]{Weis}, the constant $C$ in \eqref{Jbound} is uniformly bounded in $T\in (0,T_{0}]$. Then, by the definition of the extension of $J_{\pm}$ as a bounded map from $\mathcal{D}(\overline{\mp iA+B})$ to $L^{p}(0,T;X)$ (see the proof of Theorem \ref{LASEWP}), the same holds true for $C_{\pm}$.
\end{proof}

Based on Theorem \ref{AWEWP} and a standard Banach fixed-point argument, we establish the following.

\begin{theorem}\label{ThNLW}
Let $p\in (1,\infty)$, $c>0$, $0<T\leq T_{0}$, $X$ be a UMD complex Banach space, and let $A\in \mathcal{P}(0)\cap\mathcal{RZ}_{c}$ in $X$. Consider the problem
\begin{eqnarray}\label{NLW1}
u''(t)+A^{2}u(t) &=& F(u(t),t), \quad t\in (0,T),\\\label{NLW2}
u(0)\, \, =\, \, u'(0) &=& 0,
\end{eqnarray}
where $F(x,t)$ is a polynomial in $x$ with 
$$
F(0,\cdot)\in W_{0}^{2,p}(0,T_{0};X)\cap L^{p}(0,T_{0};\mathcal{D}(A^{2})) \cap W_{0}^{1,p}(0,T_{0};\mathcal{D}(A))
$$ 
and $t$-dependent higher-order coefficients that belong to $W_{0}^{2,\infty}(0,T_{0};\mathcal{D}(A^{2}))\cup W_{0}^{2,\infty}(0,T_{0};\mathbb{C})$. Assume that:\\
{\bf (i)} $E_{2}^{+}\hookrightarrow W_{0}^{2,\infty}(0,T;\mathcal{D}(A^{2}))$, with the norm of the embedding independent of $T\in (0,T_{0}]$.\\
{\bf (ii)} The space $\mathcal{D}(A^{2})$, up to an equivalent norm, is a Banach algebra.\\
{\bf (iii)} $\|F(0,\cdot)\|_{E_{0}^{+}}$ is uniformly bounded in $T\in (0,T_{0}]$.\\
Then, there exist a $T\in(0,T_{0}]$ and a unique $u\in E_{2}^{+}$ satisfying \eqref{NLW1}-\eqref{NLW2}. Furthermore, $u$ satisfies the regularity $u\in C^{1}([0,T];\mathcal{D}(A^{2}))$, and there exists a $C>0$, depending only on $X$, $A$, $c$, $p$, and $T_{0}$, such that
\begin{equation}\label{stabilF}
\|u\|_{E_{2}^{+}}\leq C\|F(u,\cdot)\|_{E_{0}^{+}}.
\end{equation}
\end{theorem}
\begin{proof}
Since $W_{0}^{2,p}(0,T;X)\cap L^{p}(0,T;\mathcal{D}(A^{2})) \cap W_{0}^{1,p}(0,T;\mathcal{D}(A))\subset E_{0}^{+}$, according to Theorem \ref{AWEWP}, let $w\in E_{2}^{+}$ be the unique solution of 
$$
(\overline{-iA+B} \, \overline{iA+B}w)(t)=F(0,t), \quad t\in [0,T).
$$
By (ii), the space $W_{0}^{2,\infty}(0,T;\mathcal{D}(A^{2}))$ is, up to an equivalent norm, a Banach algebra. In particular, for any $\xi\in W_{0}^{2,\infty}(0,T;\mathcal{D}(A^{2}))$ and $k\in\mathbb{N}$, we have that $\xi^{k}\in W^{2,\infty}(0,T;\mathcal{D}(A^{2}))$, with 
$$
(\xi^{k})_{t}=k\xi^{k-1}\xi_{t} \quad \text{and} \quad (\xi^{k})_{tt}=k\xi^{k-1}\xi_{tt}+k(k-1)\xi^{k-2}(\xi_{t})^{2}
$$ 
for almost all $t\in [0,T)$. Hence, due to (i), we have that 
\begin{equation}\label{DifFreg}
F(v(\cdot),\cdot)-F(0,\cdot)\in W_{0}^{2,\infty}(0,T;\mathcal{D}(A^{2})) 
\end{equation}
for any $v\in E_{2}^{+}$. Therefore, since $W_{0}^{2,\infty}(0,T;\mathcal{D}(A^{2}))\subset E_{0}^{+}$, we have that $F(v(\cdot),\cdot)\in E_{0}^{+}$ whenever $v\in E_{2}^{+}$. Based on this fact, let $r>0$ and consider the map 
$$
\{v\in E_{2}^{+} \, |\, \|v-w\|_{E_{2}^{+}}\leq r\}=U\ni v \mapsto g(v)\in E_{2}^{+},
$$
where $g(v)=h$ is, according to Theorem \ref{AWEWP}, the solution of 
$$
(\overline{-iA+B} \, \overline{iA+B}h)(t)=F(v(t),t), \quad t\in [0,T).
$$

Let $v,v_{1},v_{2}\in U$. For almost all $t\in [0,T)$, we have
\begin{eqnarray*}
\|v^{k}(t)\|_{\mathcal{D}(A^{2})} &\leq &C_{0}\|v(t)\|_{\mathcal{D}(A^{2})}^{k} \\
&\leq& C_{0}(\|v-w\|_{W_{0}^{2,\infty}(0,T;\mathcal{D}(A^{2}))}+\|w\|_{W_{0}^{2,\infty}(0,T;\mathcal{D}(A^{2}))})^{k}\\
&\leq& C_{1}(r+\|F(0,\cdot)\|_{E_{0}^{+}})^{k},
\end{eqnarray*}
for certain $C_{0},C_{1}>0$ independent of $T\in (0,T_{0}]$, where we have used Lemma \ref{CTindep} and (i)-(ii). As a consequence, if $ L\in \{F, F_{x}, F_{t}, F_{xx}, F_{xt}, F_{tt}\}$, then
\begin{equation}\label{LLipschest}
\|L(v(\cdot),\cdot)-L(0,\cdot)\|_{L^{p}(0,T;\mathcal{D}(A^{2}))} \leq C_{2}T^{\frac{1}{p}},
\end{equation}
due to (i) and (iii), for certain $C_{2}>0$ independent of $T\in (0,T_{0}]$. Furthermore, by (i)-(ii), there exists a $C_{3}>0$, independent of $T\in (0,T_{0}]$, such that 
$$
\|L(v_{1}(t),t)-L(v_{2}(t),t)\|_{\mathcal{D}(A^{2})}\leq C_{3}\|v_{1}(t)-v_{2}(t)\|_{\mathcal{D}(A^{2})};
$$
here we have used the fact that 
\begin{eqnarray*}
\lefteqn{\|v_{1}^{k}(t)-v_{2}^{k}(t)\|_{\mathcal{D}(A^{2})}}\\
&\leq&C_{4} \|v_{1}(t)-v_{2}(t)\|_{\mathcal{D}(A^{2})}\|v_{1}^{k-1}(t)+\cdots+v_{2}^{k-1}(t)\|_{\mathcal{D}(A^{2})}\\
&\leq&C_{5} \|v_{1}(t)-v_{2}(t)\|_{\mathcal{D}(A^{2})}(\|v_{1}(t)\|_{\mathcal{D}(A^{2})}^{k-1}+\cdots+\|v_{2}(t)\|_{\mathcal{D}(A^{2})}^{k-1})\\
&\leq&C_{6} \|v_{1}(t)-v_{2}(t)\|_{\mathcal{D}(A^{2})},
\end{eqnarray*}
for certain $C_{4},C_{5},C_{6}>0$. Therefore
\begin{eqnarray}\nonumber
\lefteqn{\|L(v_{1}(\cdot),\cdot)-L(v_{2}(\cdot),\cdot)\|_{L^{p}(0,T;\mathcal{D}(A^{2}))}}\\\label{lipv1v2}
&\leq& C_{3}\|v_{1}-v_{2}\|_{L^{p}(0,T;\mathcal{D}(A^{2}))}\leq C_{3}T^{\frac{1}{p}}\|v_{1}-v_{2}\|_{L^{\infty}(0,T;\mathcal{D}(A^{2}))}\leq C_{7}T^{\frac{1}{p}}\|v_{1}-v_{2}\|_{E_{2}^{+}},
\end{eqnarray}
due to (i), for certain $C_{7}>0$ independent of $T\in (0,T_{0}]$. 

Returning to the map $g$, we have
$$
(\overline{-iA+B} \, \overline{iA+B}(g(v)-w))(t)=F(v(t),t),-F(0,t), \quad t\in [0,T).
$$
Hence, by Theorem \ref{AWEWP} and Lemma \ref{CTindep}, there exists a $C_{8}>0$, independent of $T\in (0,T_{0}]$, such that
\begin{eqnarray*}
\|g(v)-w\|_{E_{2}^{+}}&\leq& C_{8}\|F(v(\cdot),\cdot)-F(0,\cdot)\|_{E_{0}^{+}}\\
&\leq&C_{8}(\|F(v(\cdot),\cdot)-F(0,\cdot)\|_{L^{p}(0,T;X)}+\|\overline{iA+B}(F(v(\cdot),\cdot)-F(0,\cdot))\|_{L^{p}(0,T;X)}\\
&&+\|\overline{-iA+B}\, \overline{iA+B}(F(v(\cdot),\cdot)-F(0,\cdot))\|_{L^{p}(0,T;X)}).
\end{eqnarray*}
Due to \eqref{DifFreg}, we have
\begin{eqnarray}\nonumber
\|g(v)-w\|_{E_{2}^{+}}&\leq& C_{8}(\|F(v(\cdot),\cdot)-F(0,\cdot)\|_{L^{p}(0,T;X)}\\\nonumber
&&+\|(iA+B)(F(v(\cdot),\cdot)-F(0,\cdot))\|_{L^{p}(0,T;X)}\\\nonumber
&&+\|(-iA+B)(iA+B)(F(v(\cdot),\cdot)-F(0,\cdot))\|_{L^{p}(0,T;X)})\\\nonumber
&\leq& C_{9}(\|F(v(\cdot),\cdot)-F(0,\cdot)\|_{W_{0}^{2,p}(0,T;X)}+\|F(v(\cdot),\cdot)-F(0,\cdot)\|_{W_{0}^{1,p}(0,T;\mathcal{D}(A))}\\\nonumber
&&+\|F(v(\cdot),\cdot)-F(0,\cdot)\|_{L^{p}(0,T;\mathcal{D}(A^{2}))})\\\nonumber
&\leq& C_{10}(\|F(v(\cdot),\cdot)-F(0,\cdot)\|_{L^{p}(0,T;\mathcal{D}(A^{2}))}\\\nonumber
&&+\|F_{x}(v(\cdot),\cdot)v'(\cdot)+F_{t}(v(\cdot),\cdot)-F_{t}(0,\cdot)\|_{L^{p}(0,T;\mathcal{D}(A))}\\\nonumber
&&+\|F_{xx}(v(\cdot),\cdot)(v'(\cdot))^{2}+2F_{xt}(v(\cdot),\cdot)v'(\cdot)\\\label{Balllipshest}
&&+F_{x}(v(\cdot),\cdot)v''(\cdot)+F_{tt}(v(\cdot),\cdot)-F_{tt}(0,\cdot)\|_{L^{p}(0,T;X)}),
\end{eqnarray}
for certain $C_{9},C_{10}>0$ independent of $T\in (0,T_{0}]$. The third term on the right-hand side of the above inequality is
\begin{eqnarray*}
\lefteqn{\leq C_{11}(\|F_{xx}(v(\cdot),\cdot)-F_{xx}(0,\cdot)\|_{L^{p}(0,T;\mathcal{D}(A^{2}))}\|v'\|_{L^{\infty}(0,T;\mathcal{D}(A^{2}))}^{2}}\\
&&+\|F_{xx}(0,\cdot)\|_{L^{p}(0,T;Y)}\|v'\|_{L^{\infty}(0,T;\mathcal{D}(A^{2}))}^{2}\\
&&+\|F_{xt}(v(\cdot),\cdot)-F_{xt}(0,\cdot)\|_{L^{p}(0,T;\mathcal{D}(A^{2}))}\|v'\|_{L^{\infty}(0,T;\mathcal{D}(A^{2}))}\\
&&+\|F_{xt}(0,\cdot)\|_{L^{p}(0,T;Y)}\|v'\|_{L^{\infty}(0,T;\mathcal{D}(A^{2}))}\\
&&+\|F_{x}(v(\cdot),\cdot)-F_{x}(0,\cdot)\|_{L^{p}(0,T;\mathcal{D}(A^{2}))}\|v''\|_{L^{\infty}(0,T;\mathcal{D}(A^{2}))} \\
&&+\|F_{x}(0,\cdot)\|_{L^{p}(0,T;Y)}\|v''\|_{L^{\infty}(0,T;\mathcal{D}(A^{2}))} +\|F_{tt}(v(\cdot),\cdot)-F_{tt}(0,\cdot)\|_{L^{p}(0,T;\mathcal{D}(A^{2}))})\\
&\leq&C_{12}\big(T^{\frac{1}{p}}(\|v\|_{E_{2}^{+}}^{2} +\|v\|_{E_{2}^{+}}+1) +\|F_{xx}(0,\cdot)\|_{L^{p}(0,T;Y)}\|v\|_{E_{2}^{+}}^{2}\\
&&+(\|F_{xt}(0,\cdot)\|_{L^{p}(0,T;Y)}+\|F_{x}(0,\cdot)\|_{L^{p}(0,T;Y)})\|v\|_{E_{2}^{+}}\big)\\
&\leq&C_{12}\big(T^{\frac{1}{p}}((\|w\|_{E_{2}^{+}}+r)^{2} +\|w\|_{E_{2}^{+}}+r+1) +\|F_{xx}(0,\cdot)\|_{L^{p}(0,T;Y)}(\|w\|_{E_{2}^{+}}+r)^{2}\\
&&+(\|F_{xt}(0,\cdot)\|_{L^{p}(0,T;Y)}+\|F_{x}(0,\cdot)\|_{L^{p}(0,T;Y)})(\|w\|_{E_{2}^{+}}+r)\big)\\
&\leq&C_{13}\big(T^{\frac{1}{p}}((\|F(0,\cdot)\|_{E_{0}^{+}}+1)^{2} +\|F(0,\cdot)\|_{E_{0}^{+}}+1) \\
&&+\|F_{xx}(0,\cdot)\|_{L^{p}(0,T;Y)}(\|F(0,\cdot)\|_{E_{0}^{+}}+1)^{2}\\
&&+(\|F_{xt}(0,\cdot)\|_{L^{p}(0,T;Y)}+\|F_{x}(0,\cdot)\|_{L^{p}(0,T;Y)})(\|F(0,\cdot)\|_{E_{0}^{+}}+1)\big),
\end{eqnarray*}
for certain $C_{11},C_{12},C_{13}>0$ independent of $T\in (0,T_{0}]$, where $Y$ is either $\mathcal{D}(A^{2})$ or $\mathbb{C}$. Treating the second term on the right-hand side of \eqref{Balllipshest} similarly, by \eqref{LLipschest} and (iii) we see that, after choosing $T$ sufficiently small, we can make $\|g(v)-w\|_{E_{2}^{+}}\leq r/2$ for all $v\in U$, and hence $g$ induces a map from $U$ to itself.

Furthermore,
$$
(\overline{-iA+B} \, \overline{iA+B}(g(v_{1})-g(v_{2})))(t)=F(v_{1}(t),t)-F(v_{2}(t),t), \quad t\in [0,T).
$$
Similarly to \eqref{DifFreg}, we have
$$
F(v_{1}(\cdot),\cdot)-F(v_{2}(\cdot),\cdot)\in W_{0}^{2,\infty}(0,T;\mathcal{D}(A^{2})) \subset E_{0}^{+},
$$
so that, by Theorem \ref{AWEWP}, we get
\begin{eqnarray}\nonumber
\|g(v_{1})-g(v_{2})\|_{E_{2}^{+}}&\leq& C_{8}(\|F(v_{1}(\cdot),\cdot)-F(v_{2}(\cdot),\cdot)\|_{L^{p}(0,T;X)}\\\nonumber
&&+\|(iA+B)(F(v_{1}(\cdot),\cdot)-F(v_{2}(\cdot),\cdot))\|_{L^{p}(0,T;X)}\\\nonumber
&&+\|(-iA+B)(iA+B)(F(v_{1}(\cdot),\cdot)-F(v_{2}(\cdot),\cdot))\|_{L^{p}(0,T;X)})\\\nonumber
&\leq& C_{9}(\|F(v_{1}(\cdot),\cdot)-F(v_{2}(\cdot),\cdot)\|_{W_{0}^{2,p}(0,T;X)}\\\nonumber
&&+\|F(v_{1}(\cdot),\cdot)-F(v_{2}(\cdot),\cdot)\|_{W_{0}^{1,p}(0,T;\mathcal{D}(A))}\\\nonumber
&&+\|F(v_{1}(\cdot),\cdot)-F(v_{2}(\cdot),\cdot)\|_{L^{p}(0,T;\mathcal{D}(A^{2}))})\\\label{2ndtermLip}
&\leq&C_{14}\|F(v_{1}(\cdot),\cdot)-F(v_{2}(\cdot),\cdot)\|_{W_{0}^{2,p}(0,T;\mathcal{D}(A^{2}))},
\end{eqnarray}
for certain $C_{14}>0$. We have
\begin{eqnarray}\nonumber
\lefteqn{\|F(v_{1}(\cdot),\cdot)-F(v_{2}(\cdot),\cdot)\|_{W_{0}^{2,p}(0,T;\mathcal{D}(A^{2}))}}\\\nonumber
&\leq& C_{15}(\|F(v_{1}(\cdot),\cdot)-F(v_{2}(\cdot),\cdot)\|_{L^{p}(0,T;\mathcal{D}(A^{2}))}\\\nonumber
&&+\|F_{x}(v_{1}(\cdot),\cdot)v'_{1}(\cdot)-F_{x}(v_{2}(\cdot),\cdot) v'_{2}(\cdot)\|_{L^{p}(0,T;\mathcal{D}(A^{2}))}\\\nonumber
&&+\|F_{t}(v_{1}(\cdot),\cdot)-F_{t}(v_{2}(\cdot),\cdot) \|_{L^{p}(0,T;\mathcal{D}(A^{2}))}\\\nonumber
&&+\|F_{xx}(v_{1}(\cdot),\cdot)(v'_{1}(\cdot))^{2}-F_{xx}(v_{2}(\cdot),\cdot)(v'_{2}(\cdot))^{2}\|_{L^{p}(0,T;\mathcal{D}(A^{2}))}\\\nonumber
&&+\|F_{xt}(v_{1}(\cdot),\cdot)v'_{1}(\cdot)-F_{xt}(v_{2}(\cdot),\cdot) v'_{2}(\cdot)\|_{L^{p}(0,T;\mathcal{D}(A^{2}))}\\\nonumber
&&+\|F_{x}(v_{1}(\cdot),\cdot)v''_{1}(\cdot)-F_{x}(v_{2}(\cdot),\cdot) v''_{2}(\cdot)\|_{L^{p}(0,T;\mathcal{D}(A^{2}))}\\\label{Fxxesty}
&&+\|F_{tt}(v_{1}(\cdot),\cdot)-F_{tt}(v_{2}(\cdot),\cdot) \|_{L^{p}(0,T;\mathcal{D}(A^{2}))}),
\end{eqnarray}
for certain $C_{15}>0$ independent of $T\in (0,T_{0}]$. By Lemma \ref{CTindep}, \eqref{LLipschest}, and \eqref{lipv1v2}, we estimate
\begin{eqnarray*}
\lefteqn{\|F_{xx}(v_{1}(\cdot),\cdot)(v'_{1}(\cdot))^{2}-F_{xx}(v_{2}(\cdot),\cdot)(v'_{2}(\cdot))^{2}\|_{L^{p}(0,T;\mathcal{D}(A^{2}))}}\\
&\leq&\|F_{xx}(v_{1}(\cdot),\cdot)(v'_{1}(\cdot))^{2} -F_{xx}(v_{2}(\cdot),\cdot)(v'_{1}(\cdot))^{2}\|_{L^{p}(0,T;\mathcal{D}(A^{2}))}\\
&&+\|F_{xx}(v_{2}(\cdot),\cdot)(v'_{1}(\cdot))^{2}-F_{xx}(v_{2}(\cdot),\cdot)(v'_{2}(\cdot))^{2}\|_{L^{p}(0,T;\mathcal{D}(A^{2}))}\\
&\leq&C_{16}(\|F_{xx}(v_{1}(\cdot),\cdot) -F_{xx}(v_{2}(\cdot),\cdot)\|_{L^{p}(0,T;\mathcal{D}(A^{2}))} \|v'_{1}\|_{L^{\infty}(0,T;\mathcal{D}(A^{2}))}^{2}\\
&&+\|F_{xx}(v_{2}(\cdot),\cdot)\|_{L^{p}(0,T;Y)} \|v'_{1}+v'_{2}\|_{L^{\infty}(0,T;\mathcal{D}(A^{2}))}\|v'_{1}-v'_{2}\|_{L^{\infty}(0,T;\mathcal{D}(A^{2}))})\\
&\leq&C_{17}\big(T^{\frac{1}{p}}\|v_{1}-v_{2}\|_{E_{2}^{+}} \|v_{1}\|_{E_{2}^{+}}^{2} +(\|F_{xx}(v_{2}(\cdot),\cdot)- F_{xx}(w,\cdot)\|_{L^{p}(0,T;\mathcal{D}(A^{2}))} \\
&&+ \|F_{xx}(w,\cdot)-F_{xx}(0,\cdot)\|_{L^{p}(0,T;\mathcal{D}(A^{2}))} +\|F_{xx}(0,\cdot)\|_{L^{p}(0,T;Y)} ) \|v_{1}+v_{2}\|_{E_{2}^{+}}\|v_{1}-v_{2}\|_{E_{2}^{+}}\big)\\
&\leq&C_{18}\big(T^{\frac{1}{p}}(\|w\|_{E_{2}^{+}}+r)+T^{\frac{1}{p}}(r+1)+\|F_{xx}(0,\cdot)\|_{L^{p}(0,T;Y)} \big)(\|w\|_{E_{2}^{+}}+r)\|v_{1}-v_{2}\|_{E_{2}^{+}}\\
&\leq&C_{19}\big(T^{\frac{1}{p}}(\|F(0,\cdot)\|_{E_{0}^{+}}+1)+\|F_{xx}(0,\cdot)\|_{L^{p}(0,T;Y)} \big)(\|F(0,\cdot)\|_{E_{0}^{+}}+1)\|v_{1}-v_{2}\|_{E_{2}^{+}},
\end{eqnarray*}
for certain $C_{16},\dots,C_{19}>0$ independent of $T\in (0,T_{0}]$. Treating the other terms on the right-hand side of \eqref{Fxxesty} similarly, by \eqref{2ndtermLip}, for suficiently small $T>0$, we find that 
$$
\|g(v_{1})-g(v_{2})\|_{E_{2}^{+}}\leq \frac{1}{2}\|v_{1}-v_{2}\|_{E_{2}^{+}}, \quad \text{for all} \quad v_{1},v_{2}\in U. 
$$
Hence, by the Banach fixed point theorem, and since $r$ is arbitrary, there exists a $T>0$ and a unique $u\in E_{2}^{+}$ satisfying
$$
(\overline{-iA+B} \, \overline{iA+B}u)(t)=F(u(t),t), \quad t\in [0,T);
$$ 
note that for any $r_{0}>0$, the constants $C_{0},\dots,C_{19}$ can be chosen independent of $r\in (0,r_{0}]$. 

By (i), we have that $u, Au,Bu\in W_{0}^{1,p}(0,T;X)\cap L^{p}(0,T;\mathcal{D}(A))$, which shows that $u$ satisfies \eqref{NLW1}-\eqref{NLW2}. In addition, by \cite[Corollary L.4.6]{HNVW1}, we have
$$
u,u'\in W^{1,p}(0,T;\mathcal{D}(A^{2}))=W^{1,p}(0,T;\mathcal{D}(A^{2}))\cap L^{p}(0,T;\mathcal{D}(A^{2})) \hookrightarrow C([0,T];\mathcal{D}(A^{2})),
$$
so that $u\in C^{1}([0,T];\mathcal{D}(A^{2}))$. Finally, since $F(u(\cdot),\cdot)\in E_{0}^{+}$, by Theorem \ref{AWEWP} the problem
$$
(\overline{-iA+B} \, \overline{iA+B}\eta)(t)=F(u(t),t), \quad t\in (0,T],
$$ 
has a unique solution $\eta\in E_{2}^{+}$, which must coincide with $u$. Then, \eqref{stabilF} follows by Theorem \ref{AWEWP} and Lemma \ref{CTindep}.
\end{proof}

\begin{remark}
{\bf (i)} Let $c>0$, $X$ be a complex Banach space, and let $\Lambda\in \mathcal{RV}_{c}$ in $X$. Then, by Remark \ref{StripvsParabola}, we have that $\Lambda^{\frac{1}{2}}\in \mathcal{RZ}_{c}$. Based on this observation, and on Remark \ref{convprop} (i) and (ii), in Theorem \ref{ThNLW} we can express all the relevant data in terms of the original operator $A^{2}=\Lambda$ instead of $A$. \\
{\bf (ii)} The mixed-derivative regularity that appears in the assumptions of Theorem \ref{ThNLW} can be obtained by Lemma \ref{fracderiv} by requiring $\mathcal{BIP}$ for the operator $A^{2}=\Lambda$; see Remark \ref{convprop}(i).\\
{\bf (iii)} Since we do not have a description of the space $E_{2}^{\pm}$ in terms of vector-valued Sobolev spaces, and we do not know whether the sum $\pm iA+B$ is closed, the assumptions (i) and (ii) in Theorem \ref{ThNLW} are natural and necessary in this situation.\\
{\bf (iv)} Based on Theorem \ref{AWEWP} and following the proof of Theorem \ref{ThNLW}, we can establish an existence, uniqueness, and regularity result for a quasilinear version of \eqref{NLW1}-\eqref{NLW2}. The same holds true for a quasilinear version of \eqref{ACLSE}, based on Theorem \ref{AWEWP} or, even better, on Theorem \ref{BestScrth}. In both cases, we may consider more general forcing terms on the right hand side of the equations; the key point is that certain Fr\'echet derivatives of the quasilinear operators and of the forcing terms must be locally Lipschitz continuous. \\
{\bf (v)} Let $X$ be a Hilbert space and let $A$ be a positive definite self-adjoint linear operator in $X$. Then $A\in \mathcal{P}(0)$, and the fractional powers defined by \eqref{fracpaw} coincide with the fractional powers defined by the spectral theorem for self-adjoint operators; see, e.g., \cite[Theorem III.4.6.7]{Amann1}. In particular, by the spectral theorem for self-adjoint operators, we have that $A\in \mathcal{P}(\theta)\cap \mathcal{BIP}$ for any $\theta\in [0,\pi)$. Recall that, in a Hilbert space, $R$-boundedness is equivalent to uniform boundedness for a family of bounded operators. Thus, by the spectral theorem for self-adjoint operators, we have that $A\in\mathcal{RZ}_{c}$ for any $c>0$. Moreover, the square root $A^{\frac{1}{2}}$ of $A$ is again a positive definite self-adjoint linear operator in $X$, and hence, $A^{\frac{1}{2}}\in\mathcal{RZ}_{c}$ for any $c>0$. Therefore, by Remark \ref{StripvsParabola}, we obtain that $A\in \mathcal{RV}_{c}$ for any $c>0$. \mbox{\ } \hfill $\Diamond$
\end{remark}

\section{An application: The Klein-Gordon equation on $\mathbb{R}$}\label{Sec7}

Let $q\in(1,\infty)$, and consider the derivation operator
$$
H_{q}^{1}(\mathbb{R})=D(A)\ni u\mapsto Au=\partial_{x}u\in L^{q}(\mathbb{R}),
$$
where $x\in\mathbb{R}$. By \cite[Example I.5.4, Proposition II.1 and Corollary II.3.7]{EnNa} or \cite[pp.240]{Haase}, we have that $\pm iA \in \mathcal{S}_{\eta}$ in $ L^{q}(\mathbb{R})$ for any $\eta>0$. Therefore, by Remark \ref{StripvsParabola}, the operator 
$$
H_{q}^{2}(\mathbb{R})=D(\Delta_{0})\ni u\mapsto \Delta_{0} u=\partial_{x}^{2}u\in L^{q}(\mathbb{R})
$$
satisfies 
\begin{equation}\label{DeltaParab}
\Pi_{\eta}\subset \rho(\Delta_{0}) \quad \text{and} \quad \|(z-\Delta_{0})^{-1}\|_{\mathcal{L}(L^{q}(\mathbb{R}))}\leq \frac{C_{0}}{\sqrt{|z|}}, \quad \|A(z-\Delta_{0})^{-1}\|_{\mathcal{L}(L^{q}(\mathbb{R}))}\leq C_{0}, \quad z\in \Pi_{\eta},
\end{equation}
for some $C_{0}>0$. On the other hand, by \cite[Example 3.7.6, Remark 3.7.7(a) and Theorem 3.7.11]{ABHN}, for any $c>0$, we have $c-\Delta_{0} \in \mathcal{P}(\frac{\pi}{2})$. Additionally, by \cite[pp. 232 and Proposition 8.3.4]{Haase}, we also have $c-\Delta_{0} \in \mathcal{BIP}$. Hence, \cite[Theorem 5.6.1]{HNVW} and \cite[Theorem 15.3.9]{HNVW1} imply that
$$
\mathcal{D}((c-\Delta_{0})^{\frac{1}{2}})=[L^{q}(\mathbb{R}), H_{q}^{2}(\mathbb{R})]_{\frac{1}{2}}=H_{q}^{1}(\mathbb{R}).
$$
Thus, for any $z\in \Pi_{\eta}$, by \eqref{DeltaParab}, we have
\begin{eqnarray*}
\lefteqn{\|(c-\Delta_{0})^{\frac{1}{2}}(z+c-\Delta_{0})^{-1}\|_{\mathcal{L}(L^{q}(\mathbb{R}))} } \\
&\leq& \|(c-\Delta_{0})^{\frac{1}{2}}(A+1)^{-1}\|_{\mathcal{L}(L^{q}(\mathbb{R}))} \|(A+1)(z+c-\Delta_{0})^{-1}\|_{\mathcal{L}(L^{q}(\mathbb{R}))}\leq C_{1}
\end{eqnarray*}
for some $C_{1}>0$. The above estimate, together with \eqref{DeltaParab}, implies that $c-\Delta_{0} \in \mathcal{V}_{\eta}$.
 
Now, let $s\geq 0$ and consider the operator
$$
H_{q}^{s+2}(\mathbb{R})=D(\Delta_{s})\ni u\mapsto \Delta_{s} u=\partial_{x}^{2}u\in H_{q}^{s}(\mathbb{R}).
$$
We will show that
\begin{equation}\label{restrrtesolv}
\rho(\Delta_{0}-c)\subseteq \rho(\Delta_{s}-c) \quad \text{and} \quad (\lambda+c-\Delta_{s})^{-1}=(\lambda+c-\Delta_{0})^{-1}|_{H_{q}^{s}(\mathbb{R})} \quad \text{for any} \quad \lambda\in \rho(\Delta_{0}-c).
\end{equation}
To this end, it suffices to show that, for any $ \lambda\in \rho(\Delta_{0}-c)$, the identities 
$$
(\lambda+c-\Delta_{0})^{-1}(\lambda+c-\Delta_{s})=I \quad \text{on} \quad H_{q}^{s+2}(\mathbb{R}) 
$$
and 
$$
(\lambda+c-\Delta_{s})(\lambda+c-\Delta_{0})^{-1}=I \quad \text{on} \quad H_{q}^{s}(\mathbb{R})
$$
hold. The first identity is trivial. For the second identity, assume first that $s\in(0,2]$, and let $v\in H_{q}^{2}(\mathbb{R})$ be such that $(\lambda+c-\Delta_{s})v \in H_{q}^{s}(\mathbb{R})$. Then $\Delta_{s} v \in H_{q}^{s}(\mathbb{R})$, i.e., $v$ lies in the maximal domain of $\Delta_{s}$ in $H_{q}^{s}(\mathbb{R})$, which equals $H_{q}^{s+2}(\mathbb{R})$. Iterating this argument, we obtain the result for any $s\geq0$. 

Let $k\in \mathbb{N}_{0}$, and recall that $\|(c-\Delta_{k})\cdot\|_{H_{q}^{k}(\mathbb{R})}$ is an equivalent norm for $H_{q}^{k+2}(\mathbb{R})$ (see, e.g., \cite[Definition 5.6.2]{HNVW}). For any $\lambda\in \rho(\Delta_{0}-c)$ and $u\in H_{q}^{k+2}(\mathbb{R})$, using \eqref{restrrtesolv} we estimate 
\begin{eqnarray}\nonumber
\|(\lambda+c-\Delta_{k+2})^{-1}u\|_{H_{q}^{k+2}(\mathbb{R})} &=& \|(\lambda+c-\Delta_{k})^{-1}u\|_{H_{q}^{k+2}(\mathbb{R})} \\\nonumber
&\leq& C_{2}\|(c-\Delta_{k})(\lambda+c-\Delta_{k})^{-1}u\|_{H_{q}^{k}(\mathbb{R})} \\\nonumber
&=& C_{2}\|(\lambda+c-\Delta_{k})^{-1}(c-\Delta_{k})u\|_{H_{q}^{k}(\mathbb{R})} \\\nonumber
&\leq& C_{2}\|(\lambda+c-\Delta_{k})^{-1}\|_{\mathcal{L}(H_{q}^{k}(\mathbb{R}))} \|(c-\Delta_{k})u\|_{H_{q}^{k}(\mathbb{R})} \\\label{iterest}
&\leq& C_{3}\|(\lambda+c-\Delta_{k})^{-1}\|_{\mathcal{L}(H_{q}^{k}(\mathbb{R}))} \|u\|_{H_{q}^{k+2}(\mathbb{R})} 
\end{eqnarray}
for some $C_{2},C_{3}>0$. Therefore, by iteration, we obtain that $c-\Delta_{2\rho}\in \mathcal{P}(\frac{\pi}{2})$ for any $\rho\in \mathbb{N}_{0}$. Then, by \eqref{restrrtesolv}, and complex interpolation (see, e.g., \cite[Theorem 5.6.9]{HNVW} and \cite[Theorem 2.6]{Lunardi}), we also have $c-\Delta_{s}\in \mathcal{P}(\frac{\pi}{2})$ for any $s\geq0$. As a consequence, by \eqref{fracpaw} and \eqref{restrrtesolv}, we deduce that 
\begin{equation}\label{restrpowers}
(c-\Delta_{s})^{z}=(c-\Delta_{0})^{z}|_{H_{q}^{s}(\mathbb{R})} \quad \text{for any} \quad z\in \mathbb{C}. 
\end{equation}
Then, by \cite[Lemma III.4.7.4]{Amann1}, \eqref{restrpowers}, an estimate similar to \eqref{iterest}, iteration, and complex interpolation, we obtain $c-\Delta_{s}\in \mathcal{BIP}$ for any $s\geq0$. Similarly, \eqref{restrpowers}, an estimate similar to \eqref{iterest}, iteration, and complex interpolation show that $c-\Delta_{s}\in \mathcal{V}_{\eta}$ for any $s\geq0$.

By \cite[Theorem III.4.5.2]{Amann1} and \cite[Proposition 4.2.15]{HNVW}, each of the spaces $H_{q}^{s}(\mathbb{R})$, $s\geq0$, is a UMD space. Furthermore, $\mathcal{D}((c-\Delta_{s})^{2})=H_{q}^{s+4}(\mathbb{R})$,
$$
[H_{q}^{s}(\mathbb{R}), H_{q}^{s+4}(\mathbb{R})]_{\frac{2+\ell}{4}}=H_{q}^{s+2+\ell}(\mathbb{R}), \quad \text{and} \quad [H_{q}^{s}(\mathbb{R}), H_{q}^{s+2}(\mathbb{R})]_{\frac{\ell}{2}}=H_{q}^{s+\ell}(\mathbb{R}) 
$$
for all $\ell\in (0,1)$ and $s\geq0$, due to \cite[Theorem 5.6.9]{HNVW}. By Theorem \ref{invA2B2} and Remark \ref{convprop}, we summarize the results as follows.

\begin{corollary}[Klein-Gordon equation on $\mathbb{R}$]
Let $T,c>0$, $p,q\in(1,\infty)$, $s\geq0$, and let $f\in W_{0}^{2+\nu,p}(0,T;H_{q}^{s+\ell}(\mathbb{R})) \cap W_{0}^{\nu,p}(0,T;H_{q}^{s+2+\ell}(\mathbb{R}))$ for some $\nu,\ell\in(0,1]$ such that $\nu+\ell>1$. Then, there exists a unique $u\in W_{0}^{2,p}(0,T;H_{q}^{s}(\mathbb{R})) \cap L^{p}(0,T;H_{q}^{s+2}(\mathbb{R}))$ satisfying
$$
u_{tt}(t,x)-u_{xx}(t,x)+cu(t,x)=f(t,x), \quad (t,x)\in (0,T)\times \mathbb{R}, \quad u(0,x)=u_{t}(0,x)=0, \quad x\in \mathbb{R}.
$$
Moreover, the solution $u$ depends continuously on $f$.
\end{corollary}

\end{document}